\documentclass[11pt,a4paper,reqno]{amsart}
\usepackage[english]{babel}
\usepackage[T1]{fontenc}
\usepackage{verbatim}
\usepackage{palatino}
\usepackage{amsmath}
\usepackage{mathabx}
\usepackage{amssymb}
\usepackage{tcolorbox}
\usepackage{amsthm}
\usepackage{amsfonts}
\usepackage{graphicx}
\usepackage{tikz}
\usetikzlibrary{calc,angles,quotes,arrows.meta, decorations.pathreplacing}
\usepackage{esint}
\usepackage{color}
\usepackage[dvipsnames]{xcolor}
\usepackage{mathtools}
\usepackage{overpic}

\usepackage[colorlinks = true, citecolor = black]{hyperref}
\usepackage{cleveref}
\author{Benjamin Jaye and Rahul Sethi}
\title[Uncertainty Principle]{Quantitative Uniqueness and Rough Damping on $\mathbb T^2$}
\address{School of Mathematics\\ Georgia Institute of Technology \\ Atlanta, GA, USA} \email{bjaye3@gatech.edu}
\email{rahul.sethi@math.gatech.edu}

\date{\today}

\newcommand{\C}{\mathbb C}

\newcommand{\h}{\mathcal H}

\newcommand{\N}{\mathbb N}
\newcommand{\Z}{\mathbb Z}

\newcommand{\ep}{\varepsilon}
\newcommand{\Q}{\mathbb Q}
\newcommand{\R}{\mathbb R}
\newcommand{\D}{\mathbb D}
\newcommand{\T}{\mathbb T}

\renewcommand{\ge}{\geqslant}
\renewcommand{\le}{\leqslant}

\numberwithin{equation}{section}

\theoremstyle{plain}

\newtheorem{theorem}{Theorem}[section]     
\newtheorem{conjecture}[theorem]{Conjecture}
\newtheorem{lemma}[theorem]{Lemma}

\newtheorem{corollary}[theorem]{Corollary}
\newtheorem{proposition}[theorem]{Proposition}

\newtheorem*{"thm"}{"Theorem"}

\newtheorem*{lemma*}{Lemma}

\theoremstyle{definition}

\theoremstyle{remark}
\newtheorem{remark}[equation]{Remark}

\newcommand{\nref}[1]{(\hyperref[#1]{#1})}

\begin{document}
\begin{abstract}
    Motivated by a conjecture of Burq and G\'erard, we investigate quantitative uniqueness principles for functions on $\mathbb T^2$ whose Fourier spectra lie in fixed-width annuli. We obtain observability estimates uniform in the radius, under a mild Sobolev regularity condition, and the generalized geometric control condition (GGCC) on the damping function. This leads to exponential decay for the damped wave equation with rough damping. We also establish an analogous uncertainty principle for spectra near dilates of convex polygonal boundaries, where no Sobolev regularity of the damping is required beyond $L^\infty$.
\end{abstract}

\maketitle

\section{Introduction}

Let $\mathbb T^2:=\mathbb R^2/\mathbb Z^2$, and let $\Delta$ be the Laplace operator on $\mathbb T^2$. We consider the damped wave equation
\begin{equation}
    \label{eq:damped-wave-introduction}
    \partial_t^2u-\Delta u+\gamma(x)\partial_tu=0,
    \qquad
    (u,\partial_tu)|_{t=0}=(u_0,u_1)\in H^1(\mathbb T^2)\times L^2(\mathbb T^2),
\end{equation}
for non-negative $\gamma\in L^\infty(\mathbb T^2)$. The associated energy
$$E_u(t) := \frac12\int_{\mathbb T^2} \left( |\nabla u(x,t)|^2+|\partial_tu(x,t)|^2
\right)\,dx$$
satisfies
$$\frac{d}{dt}E_u(t)
= -\int_{\mathbb T^2}\gamma(x)|\partial_tu(x,t)|^2\,dx
\le 0,$$
and so the damping function explains the decay of the energy. A fundamental problem in control theory is to determine when the energy decays uniformly and exponentially: when do there exist constants $C,c>0$ such that
\begin{equation}
    \label{eq:exponential-stabilization-introduction}
    E_u(t)\le  Ce^{-ct}E_u(0),
\end{equation}
for all $t\ge 0$, for every solution of \eqref{eq:damped-wave-introduction}? This is called uniform stabilization.

The geodesics on $\mathbb T^2$ are
$$\mathcal G_{x,\nu,L} := \{x+t\nu:0\le t\le L\},$$
for $x\in\mathbb T^2$ and $\nu\in S^1$, with addition understood modulo $\mathbb Z^2$. When the damping function $\gamma$ is continuous, uniform stabilization is characterized by the classical geometric control condition (GCC): there exist $L,c_0>0$ such that
\begin{equation}
\label{eq:classical-GCC-intro}
    \inf_{x\in\mathbb T^2}\inf_{\nu\in S^1}
\int_0^L \gamma(x+t\nu)\,dt \ge c_0.
\end{equation}

Geometric control conditions have long played a role in control theory for hyperbolic equations, for instance see Bardos, Lebeau, Rauch \cite{b-l-r, zworski-book, rauch74}, and various uncertainty principles have played an essential role in control theory \cite{bourgain18,jaffard01,zworski-book,suzuki-inami,suzuki-damping}. See also \cite{folland97,havin12} for general background on uncertainty principles.

The situation is more subtle for rough damping functions, since elements of $L^\infty(\mathbb T^2)$ are defined only up to sets of measure zero, and formulations of geometric control in terms of individual geodesics require a refinement in this setting. To this end, Burq and G\'erard introduced a generalized geometric control condition formulated in terms of averages over shrinking tubular neighborhoods of geodesics \cite{BurqGerard2020}. Writing
$$\Gamma_{x,\nu,L,\delta} := \left\{y\in\mathbb T^2: \operatorname{dist}(y,\mathcal G_{x,\nu,L})<\delta \right\},$$
we say that $\gamma$ satisfies the \emph{generalized geometric control condition} (GGCC) if there exist $L,c_0>0$ such that
\begin{equation}
\label{eq:GGCC-intro}
\liminf_{\delta\to0} \inf_{x\in\mathbb T^2}\inf_{\nu\in S^1}
\frac{1}{|\Gamma_{x,\nu,L,\delta}|} \int_{\Gamma_{x,\nu,L,\delta}}\gamma(y)\,dy \ge c_0.
\end{equation}
In contrast with the preceding conditions, GGCC depends only on the $L^\infty$ equivalence class of $\gamma$.

Burq and G\'erard proved that the corresponding GGCC is necessary for uniform stabilization on every smooth compact Riemannian manifold without boundary and for every non-negative $L^\infty$ damping function \cite{BurqGerard2020}. In the same work, they focused on $\mathbb T^2$ and proved that GGCC is also sufficient when the damping is a finite positive linear combination of characteristic functions of polygons. This two-dimensional result was subsequently generalized by Rouveyrol to a sufficient geometric condition for polyhedral damping on higher-dimensional flat tori \cite{Rouveyrol2024}. Burq and G\'erard conjectured that the polygonal structure is unnecessary; on $\mathbb T^2$, their conjecture takes the following form.

\begin{conjecture}[Burq--G\'erard]
\label{conj:burq-gerard}
Let $\gamma\in L^\infty(\mathbb T^2)$ be non-negative. Then the damped wave equation \eqref{eq:damped-wave-introduction} is uniformly stabilized if and only if $\gamma$ satisfies the GGCC \eqref{eq:GGCC-intro}.
\end{conjecture}

Burq and Joly proved uniform stabilization for the damped wave equation on $\mathbb R^d$ under a geometric control condition \cite{burq16}, assuming that the damping function is uniformly continuous. A complementary Fourier-analytic approach to geometric control has emerged in recent years; see also \cite{egidi18} for related connections between Logvinenko--Sereda estimates and controllability of the heat equation. Green subsequently gave a different proof of the one-dimensional case using a Logvinenko--Sereda uncertainty principle \cite{green2019decay, logvinenko74, paneah61}. Green, Mitkovski, and the first author developed this point of view in higher dimensions \cite{green-jaye-mitkovski}. They established an uncertainty principle for functions with spectra in annuli. In particular, if $E\subset\mathbb R^d$ satisfies the $1$-GCC and
$$
\operatorname{supp}\widehat f
\subset \left\{\xi\in\mathbb R^d:
R\le |\xi|\le R+1\right\},
$$
then,
$$\|f\|_{L^2(\mathbb R^d)} \le C_{\delta,E}\, \|f\|_{L^2(E_\delta)},$$
for every $\delta>0$, where $E_\delta$ is the $\delta$-neighborhood of $E$ and the constant is independent of $R$. This estimate was applied to obtain energy decay for damped fractional wave equations in $\R^d$. 

A natural question left open is whether the neighborhood $E_\delta$ can be replaced with $E$. A positive answer would remove the uniform continuity assumption from the stabilization theorem of Burq and Joly. The radial case of this question was subsequently resolved by the authors using a Logvinenko-Sereda uncertainty principle for the Fourier--Bessel transform \cite{jaye-sethi, gj}. This observability result was also shown independently by Wei, Duan, and Xu \cite{WeiDuanXu2026}. Related connections between geometric control and uncertainty principles have recently also appeared in the study of Schrödinger observability \cite{GreenKleinhenz2026}, where the authors introduce the comb GCC condition.

In this paper, we further develop this perspective in the specific setting of $\T^2$.

Our first main result is a quantitative uniqueness result for functions with Fourier support in annuli whose width may shrink as the radius approaches infinity. In particular, for $0\le\alpha\le\frac12$, define
$$ A_{R,\alpha} := \left\{\xi\in\mathbb R^2:
R\le|\xi|\le R+R^{-\alpha} \right\}.$$

\begin{tcolorbox}[colback=SeaGreen!10, colframe=SeaGreen!80, boxrule=0.5mm, left=2mm, right=2mm, boxsep=1mm, arc=1mm]
\begin{theorem}
\label{thm:annular-introduction}
Let $0\le\alpha\le\frac12$, and set $s:=\frac12-\alpha.$ Suppose that $\gamma\in H^s(\mathbb T^2)\cap L^\infty(\mathbb T^2)$ is non-negative and satisfies the generalized geometric control condition (GGCC): there exist $L,c_0>0$ such that
$$
\liminf_{\delta\to0}
\inf_{x\in\mathbb T^2} \inf_{\nu \in S^1}
\frac{1}{|\Gamma_{x,\nu,L,\delta}|}
\int_{\Gamma_{x,\nu,L,\delta}}\gamma(y)\,dy
\ge c_0.
$$
Then there exists a constant $C>0$, independent of $R$, such that
\begin{equation}
\label{eq:annular-main-introduction}
\|f\|_{L^2(\mathbb T^2)}^2 \le C\int_{\mathbb T^2}|f(x)|^2\, \gamma(x)\,dx,
\end{equation}
whenever
$$\operatorname{supp}\widehat f
\subset A_{R,\alpha}\cap\mathbb Z^2.$$
\end{theorem}
\end{tcolorbox}

For $\alpha=0$, the hypothesis is $\gamma\in H^{1/2}(\T^2) \cap L^\infty(\T^2)$. As the annulus becomes thinner, the
required regularity decreases. At the endpoint $\alpha=\frac12$, the Sobolev hypothesis becomes automatic under the assumed boundedness.

A proof of Theorem \ref{thm:annular-introduction} in the case $\alpha=0$ for non-negative $\gamma\in L^\infty(\mathbb T^2)$ satisfying GGCC would establish the sufficiency direction of Conjecture \ref{conj:burq-gerard} without any additional regularity assumption. In Theorem \ref{thm:polygon-introduction}, we obtain the corresponding result for frequency sets near dilates of polygons.

As in \cite{malhi2020energy, green2019decay, green-jaye-mitkovski}, we can derive from Theorem \ref{thm:annular-introduction} a resolvent estimate from which semigroup theory \cite{gearhart78, pruss84} yields the following consequence. For more details, see the Appendix.

\begin{corollary}[Exponential stabilization]
Let $\gamma\in H^{1/2}(\mathbb T^2)\cap L^\infty(\mathbb T^2)$ be non-negative and satisfy the GGCC \eqref{eq:GGCC-intro}. Then the energy of every solution of
$$
\partial_t^2u-\Delta u+\gamma(x)\partial_tu=0,
$$
on $\T^2$ decays exponentially: there exist constants $C,c>0$ such that $$E_u(t)\le Ce^{-ct}E_u(0),$$
for $t \ge 0$.
\end{corollary}

This gives a sufficient condition for uniform stabilization for a class of rough damping functions that is not covered by existing results. To illustrate the roughness permitted by this result, in Section \ref{section:rough-damping-example} we construct a non-negative function $\gamma\in H^{1/2}(\mathbb T^2)\cap L^\infty(\mathbb T^2)$ satisfying the stronger integral GCC for which $\{\gamma>0\}$ has positive measure but is nowhere dense. In particular, $\gamma$ is not bounded below by a positive constant on any nonempty open subset of $\mathbb T^2$, so the resulting stabilization cannot be reduced to observability from an open damping region.

Our second main result illustrates more explicitly how the geometry of the Fourier support determines the directions in which geometric control is required: it is a quantitative uniqueness result for functions with Fourier support near polygonal boundaries. For annular spectra, all directions occur, and accordingly Theorem \ref{thm:annular-introduction} assumes the full GGCC. On the other hand, only finitely many directions are relevant for polygonal frequency sets: those normal to its sides. For a set of directions $\Theta\subset S^1$, we say that $\gamma$ satisfies GGCC in the directions $\Theta$ if \eqref{eq:GGCC-intro} holds with the infimum in $\nu$ restricted to $\Theta$. Let $P\subset\mathbb R^2$ be a convex polygon, let $\Theta$ be the set of unit normals to its sides, and define
$$\Lambda_R(P) := \left\{\xi\in\mathbb Z^2: \operatorname{dist}(\xi,R\partial P)\le 1 \right\}.$$

\begin{tcolorbox}[colback=SeaGreen!10, colframe=SeaGreen!80, boxrule=0.5mm, left=2mm, right=2mm, boxsep=1mm, arc=1mm.]
\begin{theorem}
\label{thm:polygon-introduction}
Let $P\subset\mathbb R^2$ be a convex polygon, and let
$\gamma\in L^\infty(\mathbb T^2)$ be non-negative. Suppose that $\gamma$ satisfies GGCC in the directions $\Theta$: there exist $L,c_0>0$ such that
$$
\liminf_{\delta\to0}
\inf_{x\in\mathbb T^2} \inf_{\nu \in \Theta}
\frac{1}{|\Gamma_{x,\nu,L,\delta}|}
\int_{\Gamma_{x,\nu,L,\delta}}\gamma(y)\,dy
\ge c_0.
$$

Then there exists a constant $C>0$, depending on $P$ and $\gamma$ but independent of $R$, such that 
$$\|f\|_{L^2(\mathbb T^2)}^2 \le C\int_{\mathbb T^2}|f(x)|^2\gamma(x)\,dx$$
for every $R\geq 1$ and every $f\in L^2(\mathbb T^2)$ satisfying $$\operatorname{supp}\widehat f\subset\Lambda_R(P).$$
\end{theorem}
\end{tcolorbox}

Observe that no Sobolev regularity of the damping is needed beyond $\gamma\in L^\infty$ in Theorem \ref{thm:polygon-introduction}.

We conclude the introduction with a brief sketch of the basic ideas behind the proofs of these theorems.  A recurring tool in our work is a Logvinenko--Sereda inequality for functions whose Fourier support lies in two parallel strips. The important feature is that the observability constant is independent of the separation between the strips. The analogous result for a single strip was proved in \cite{green-jaye-mitkovski}, but it is crucial for us that we can handle antipodal pairs of strips simultaneously.  The proof of this estimate is a modification of ideas originating in Kovrizhkin \cite{k1}, where the classical Logvinenko-Sereda theorem is extended to Fourier spectra supported in a finite union of intervals.

In the case of Theorem \ref{thm:annular-introduction}, we decompose the annulus into antipodal pairs of strips of constant width and length $\sqrt{R}$. Using the pigeonhole principle, we may select strip-pairs that are separated by a small constant multiple of $\sqrt{R}$ with only a small norm loss. The two-strip inequality applies uniformly to each retained pair. The enforced separation enables  lattice-point counting to effectively limit the number of frequency pairs that can contribute to a fixed frequency difference, while the $H^{1/2-\alpha}$ regularity of $\gamma$ controls the resulting sum.

For polygonal spectra, we first remove fixed neighborhoods of the vertices and group the remaining side pieces according to their directions. The two-strip inequality controls each group, while the $L^2$ Fourier tail of $\gamma$ ensures that the interactions between non-parallel sides becomes small. Thus, any failure of observability must leave a noticeable amount of Fourier mass near some vertex. Finally, a compactness argument resting on a uniqueness result for cone-supported Fourier series rules this out.

\section{A Logvinenko-Sereda Inequality for Parallel Strips}
\label{section:strips}

We now prove the two-strip uncertainty principle described in the
Introduction. We use throughout the notation
$\mathcal G_{x,\nu,L}$ and $\Gamma_{x,\nu,L,\delta}$ introduced above for the GGCC condition.

\begin{tcolorbox}[colback=SeaGreen!10, colframe=SeaGreen!80, boxrule=0.5mm, left=2mm, right=2mm, boxsep=1mm, arc=1mm.]
\begin{theorem}[Two-strip theorem]
\label{two-strip-theorem}
Fix a strip width $W<\infty$. Let $\nu\in S^1$, and suppose that $\Sigma_1,\Sigma_2\subset \mathbb R^2$ are two strips parallel to $\nu^\perp$, each of width at most $W$. Assume that $\gamma\in L^\infty(\mathbb T^2)$ is non-negative and satisfies the following generalized geometric control condition (GGCC) in the
direction $\nu$: there exist $L,c_0>0$ such that
$$
\liminf_{\delta\to0}
\inf_{x\in\mathbb T^2}
\frac{1}{|\Gamma_{x,\nu,L,\delta}|}
\int_{\Gamma_{x,\nu,L,\delta}}\gamma(y)\,dy
\ge c_0.
$$
Then there exists a constant $C_{\mathrm{str}} < \infty$ such that every $h \in L^2(\T^2)$ satisfying $\operatorname{supp}\widehat h\subset \Sigma_1\cup \Sigma_2$ obeys
\begin{equation}
    \label{eq:two-strip-theorem}
\|h\|_{L^2(\mathbb T^2)}^2
   \le C_{\mathrm{str}}\int_{\mathbb T^2}|h(x)|^2\, \gamma(x)\,dx.
\end{equation}
In particular, the constant $C_{\mathrm{str}}$ is independent of the separation between $\Sigma_1$ and $\Sigma_2$.
\end{theorem}
\end{tcolorbox}

The proof of Theorem \ref{two-strip-theorem} is based on a high-degree Taylor approximation argument, inspired by Kovrizhkin \cite{k1}, and uses Nazarov's Tur\'an inequality \cite{nazarov} as stated in Lemma 3 of \cite{k1}:

\begin{lemma}[Nazarov's Tur\'an Inequality]
\label{nazarov-Tur\'an-kovrizhkin}
    If $r(x) = \sum_{k=1}^N p_k(x) e^{2\pi i\lambda_k x}$, where $p_k(x)$ is a polynomial of degree $\le M-1$ and $E \subset I$ is measurable with $|E| > 0$, then
    \begin{equation}
    \label{nazarov-Tur\'an-kovrizhkin-eqn}
        \|r\|_{L^p(I)} \le \left(\frac{C|I|}{|E|}\right)^{NM - \frac{p-1}{p}}\, \|r\|_{L^p(E)}.
    \end{equation}
\end{lemma}

First, we show a Bernstein inequality: after a modulation, the restriction of $h_j$ to a line in direction $\nu$ has derivatives controlled only by the strip width $W$, independent of the location of the strip. Thus each component can be approximated on $[0,L]$ by a polynomial of high degree. Recombining the two components gives an exponential polynomial with two frequencies. Finally, Lemma \ref{nazarov-Tur\'an-kovrizhkin} controls such exponential polynomials from their values on the relatively dense set, and the Taylor remainder is made small by choosing the degree large enough.

We note that the classical GCC stated in integral form (\ref{eq:classical-GCC-intro}) is equivalent to the super-level set GCC: there exist $\theta, \eta, L > 0$ such that
\begin{equation}
    \label{eq:GCC-super-level}
\left| \left\{ t\in[0,L]: \gamma(x+t\nu)\ge\eta
\right\} \right| \ge \theta L 
\end{equation}
for every $x\in\mathbb T^2$ and every $\nu\in S^1$. To apply Lemma \ref{nazarov-Tur\'an-kovrizhkin} efficiently, it is convenient to work with $\gamma$ satisfying GCC \eqref{eq:GCC-super-level} rather than GGCC \eqref{eq:GGCC-intro}. Indeed, GCC guarantees that the set $\{t\in [0,L]: \gamma(x+t\nu)\ge \eta\}$ has measure at least $\theta L$, and this set can be used directly as $E$ in \eqref{nazarov-Tur\'an-kovrizhkin-eqn}. We bridge this gap through the following observation: $\gamma$ satisfies GGCC in direction $\nu$ if and only if its mollifications $\gamma_\ep$ satisfy GCC in direction $\nu$, with GCC constants that may be chosen uniformly for all sufficiently small $\ep$. We record this equivalence below. 

\begin{lemma}
\label{lemma:ggcc-mollification}
Fix $\nu\in S^1$, and let $\gamma\in L^\infty(\mathbb T^2)$ be non-negative.
For $x\in\mathbb T^2$, $\delta>0$, and $L>0$, set
$$ \mathcal{G}_{x,\nu,L} := \{x+t\nu:0\le t\le L\}$$
and
$$\Gamma_{x,\nu,L,\delta}
:= \left\{y\in\mathbb T^2: \operatorname{dist}(y,\mathcal{G}_{x,\nu,L})<\delta \right\}.$$ 
Let $\phi\in C_c^\infty(\mathbb R^2)$ be non-negative and radial with $\|\phi\|_{L^1(\R^2)} = 1$ and $\operatorname{supp} \phi \subset B(0,1)$. Suppose that $\phi(z)\ge c_\phi>0$ for $|z|\le \frac12$. For sufficiently small $\ep>0$, let $\phi_\ep(x) := \ep^{-2} \phi(\frac{x}{\ep})$ and define $\gamma_\ep:=\gamma*\phi_\ep.$ Then the following conditions are equivalent:

\begin{enumerate}
\item There exist $L,c_0>0$ such that,
$$
\liminf_{\delta\to0} \inf_{x\in\mathbb T^2}
\frac{1}{|\Gamma_{x,\nu,L,\delta}|} \int_{\Gamma_{x,\nu,L,\delta}}\gamma(y)\,dy \ge c_0.
$$

\item There exist $\theta,\eta,L>0$ and $\ep_0>0$ such that, for
every $0<\ep<\ep_0$ and every $x\in\mathbb T^2$,
$$\left| \left\{ t\in[0,L]: \gamma_\ep(x+t\nu)\ge\eta
\right\} \right| \ge \theta L.$$

\end{enumerate}
\end{lemma}

\begin{proof}
Observe that $|\Gamma_{x,\nu,L,\delta}| \approx_L \delta$, for every $x\in\mathbb T^2, \nu \in S^1$ and every sufficiently small $\delta>0$. The constants in $\approx_L$ are independent of $x$ and $\nu$. We now compare tube averages of $\gamma$ with line integrals of $\gamma_\ep$. Define
$$K_{\ep,x,\nu}(y) := \int_0^L \phi_\ep(x+t\nu-y)\,dt.$$
By Fubini's theorem,
$$ \int_0^L \gamma_\ep(x+t\nu)\,dt = \int_{\mathbb T^2} \gamma(y)K_{\ep,x,\nu}(y)\,dy.$$

We claim that there are constants $c,C_L>0$ such that
$$
\frac{c}{\ep}
\mathbf 1_{\Gamma_{x,\nu,L,\ep/4}}(y)
\le K_{\ep,x,\nu}(y) \le \frac{C_L}{\ep} \mathbf 1_{\Gamma_{x,\nu,L,\ep}}(y)$$
for every sufficiently small $\ep>0$.

For the lower bound, consider $y\in\Gamma_{x,\nu,L,\ep/4}.$ Then there is some $t_0\in[0,L]$ such that $|x+t_0\nu - y|<\frac{\ep}{4}.$ The interval $$ [0,L]\cap \left[t_0-\frac{\ep}{4}, t_0+\frac{\ep}{4}\right] $$ has length at least $\ep/4$, and for every $t$ in this interval, $|x+t\nu - y| < \frac{\ep}{2}.$ This implies $\phi_\ep(x+t\nu-y) \ge c_\phi\, \ep^{-2},$ and therefore
$$ K_{\ep,x,\nu}(y)
\ge \frac{c_\phi}{4\ep}.
$$

For the upper bound, note that $K_{\ep,x,\nu}(y)=0$ unless $y\in\Gamma_{x,\nu,L,\ep}$ as $\operatorname{supp} \phi \subset B(0,1)$. Moreover,
$$K_{\ep,x,\nu}(y) \le \|\phi\|_\infty\, \ep^{-2}
\left| \left\{t\in[0,L]: \operatorname{dist}(x+t\nu,y)<\ep
\right\} \right|.$$
The last set has measure $\lesssim_L \ep$. Thus,
$$
K_{\ep,x,\nu}(y)
\le
\frac{C_L}{\ep}\,
\mathbf 1_{\Gamma_{x,\nu,L,\ep}}(y).
$$

We first prove that $(1)$ implies $(2)$. Suppose that $\gamma$ satisfies
GGCC in the direction $\nu$. By the definition of the liminf, there is $\delta_0>0$ such that
$$\int_{\Gamma_{x,\nu,L,\delta}}\gamma(y)\,dy
\ge \frac{c_0}{2} |\Gamma_{x,\nu,L,\delta}|,$$
for every $x\in\mathbb T^2$ and every $\delta\in (0, \delta_0)$. Using the lower bound for $K_{\ep,x,\nu}$ we obtain
$$\int_0^L\gamma_\ep(x+t\nu)\,dt = \int_{\mathbb T^2} \gamma(y)K_{\ep,x,\nu}(y)\,dy \ge \frac{c}{\ep}
\int_{\Gamma_{x,\nu,L,\ep/4}}\gamma(y)\,dy \ge \frac{cc_0}{2\ep} |\Gamma_{x,\nu,L,\ep/4}| \ge m, $$
where $m>0$ is independent of $x$ and of all sufficiently small
$\ep$. Set $M:=\|\gamma\|_{L^\infty(\mathbb T^2)}.$ Then, $0\le \gamma_\ep \le M$. Let $\eta:=\frac{m}{2L}$ and $E_{x,\ep} := \left\{ t\in[0,L]: \gamma_\ep(x+t\nu)\ge\eta \right\}.$ Then,
$$m \le \int_0^L\gamma_\ep(x+t\nu)\,dt \le \eta L+M|E_{x,\ep}| =
\frac{m}{2}+M|E_{x,\ep}|.$$
It follows that
$|E_{x,\ep}| \ge \frac{m}{2M}.$ Thus, setting
$\theta:=\frac{m}{2ML},$ we obtain $$|E_{x,\ep}|\ge\theta L$$ for every $x\in\mathbb T^2$ and every sufficiently small
$\ep>0$. This proves $(2)$. Conversely, suppose that $(2)$ holds. Then, 
$$ \int_0^L\gamma_\ep(x+t\nu)\,dt
\ge \eta\,\theta L
=:m,$$
for every $x\in\mathbb T^2$ and every sufficiently small $\ep$.
Using the upper bound for $K_{\ep,x,\nu}$ gives
$$ m \le \int_0^L\gamma_\ep(x+t\nu)\,dt = \int_{\mathbb T^2}
\gamma(y)K_{\ep,x,\nu}(y)\,dy \le \frac{C_L}{\ep}
\int_{\Gamma_{x,\nu,L,\ep}}\gamma(y)\,dy. $$
Hence, 
$$ \int_{\Gamma_{x,\nu,L,\ep}}\gamma(y)\,dy
\ge \frac{m}{C_L}\ep.$$
Since $|\Gamma_{x,\nu,L,\ep}| \lesssim_L \ep$, we obtain
$$ \frac{1}{|\Gamma_{x,\nu,L,\ep}|}
\int_{\Gamma_{x,\nu,L,\ep}}\gamma(y)\,dy \gtrsim_{\theta,\eta, L} 1.$$
This estimate is uniform in $x$ and in all sufficiently small
$\ep$. Therefore
$$\liminf_{\delta\to0}
\inf_{x\in\mathbb T^2}
\frac{1}{|\Gamma_{x,\nu,L,\delta}|}
\int_{\Gamma_{x,\nu,L,\delta}}\gamma(y)\,dy \gtrsim_{\theta,\eta, L} 1,$$
which proves $(1)$.
\end{proof}

Finally, we begin the proof of the two-strip theorem.

\begin{proof}[Proof of Theorem \ref{two-strip-theorem}]
Using properties of approximate identities, it suffices to show (\ref{eq:two-strip-theorem}) for $\gamma_\ep$ where $\ep$ is sufficiently small, with constants independent of $\ep$. Indeed, suppose that $$\|h\|_{L^2(\mathbb T^2)}^2 \le C \int_{\mathbb T^2} |h(x)|^2\,\gamma_\ep(x)\,dx,$$
for $h$ satisfying the hypothesis of Theorem \ref{two-strip-theorem}, where $C$ is independent of $\ep$. Fubini's theorem gives $$ \int_{\mathbb T^2}|h(x)|^2\gamma_\ep(x)\,dx
= \int_{\mathbb T^2} \gamma(y)
\bigl(|h|^2*{\phi}_\ep\bigr)(y)\,dy.
$$
Since ${\phi}_\ep$ is an approximate identity, $|h|^2*{\phi}_\ep
\to |h|^2$ in $L^1(\mathbb T^2)$, giving 
$$\left| \int_{\mathbb T^2}|h|^2\gamma_\ep
- \int_{\mathbb T^2}|h|^2\gamma
\right| \le \|\gamma\|_{L^\infty}
\left\| |h|^2*{\phi}_\ep-|h|^2
\right\|_{L^1} \xrightarrow{\ep \to 0} 0.$$
Letting $\ep\to0$ therefore yields (\ref{eq:two-strip-theorem}). Hereafter, we will work to show (\ref{eq:two-strip-theorem}) for $\gamma_\ep$. In view of Lemma \ref{lemma:ggcc-mollification}, $\gamma_\ep$ satisfy the GCC (\ref{eq:GCC-super-level}) with constants independent of $\ep$ for small enough $\ep$. 

    Without loss of generality, assume $$\Sigma_j = \{\xi\in\mathbb R^2: |\xi\cdot\nu-a_j|\le W\},$$
for $j = 1,2$ where $a_1,a_2\in \R$. We first decompose the Fourier support into two disjoint pieces. Let $\Lambda_j=\operatorname{supp}\widehat h\cap \Sigma_j$ with the convention that, if the two strips overlap, lattice points in the
overlap are assigned to exactly one of the two pieces. Thus,
$$\operatorname{supp}\widehat h=\Lambda_1\sqcup \Lambda_2.$$
Write $h=h_1+h_2$ where $\widehat h_j=\widehat h\,\mathbf 1_{\Lambda_j}$. Then,
\begin{equation}
    \label{two-strip-eq1}
    \|h\|_{L^2(\mathbb T^2)}^2 = \|h_1\|_{L^2(\mathbb T^2)}^2
+ \|h_2\|_{L^2(\mathbb T^2)}^2.
\end{equation}

For $j=1,2$, define $D_j:=\nu\cdot \nabla_x-2\pi i a_j.$ If $\xi\in \Lambda_j$, then $$D_j e^{2\pi i \xi\cdot x}
= 2\pi i\,(\xi\cdot\nu-a_j)\, e^{2\pi i \xi\cdot x}.$$

Since $|\xi\cdot\nu-a_j|\le W$ on $\Lambda_j$, Plancherel's theorem gives the Bernstein inequality
$$
\|D_j^k h_j\|_{L^2(\mathbb T^2)}
   \le (2\pi W)^k\, \|h_j\|_{L^2(\mathbb T^2)}
$$
for every integer $k\ge 0$. Now, we restrict to line segments in the direction $\nu$. For $x\in \mathbb T^2$ and $0\le t\le L$, set
\[
H_x(t):=h(x+t\nu).
\]
Similarly, define
\[
g_{j,x}(t):=e^{-2\pi i a_j t}h_j(x+t\nu).
\]
Then
\[
H_x(t)
=
e^{2\pi i a_1 t}g_{1,x}(t)
+
e^{2\pi i a_2 t}g_{2,x}(t).
\]
Moreover, we can show 
\[
\frac{d^k}{dt^k}g_{j,x}(t)
=
e^{-2\pi i a_j t}(D_j^k h_j)(x+t\nu).
\]
by induction on $k\in \N$. 

Now, we set up the Taylor approximation step. Fix a degree $d\ge 0$, to be chosen later. For each $x$ and $j=1,2$, let $T_{j,x}$ be the Taylor polynomial
of $g_{j,x}$ of degree $d$ centered at $t=0$:
\[
T_{j,x}(t)
=
\sum_{m=0}^d \frac{g_{j,x}^{(m)}(0)}{m!}t^m.
\]
Write $$ g_{j,x}(t)=T_{j,x}(t)+R_{j,x}(t),$$
where the remainder term $R_{j,x}$ can be expressed in integral form as
$$ R_{j,x}(t) = \int_0^t \frac{(t-s)^d}{d!}g_{j,x}^{(d+1)}(s)\,ds.$$
Therefore, for $0\le t\le L$,
\[
|R_{j,x}(t)|
\le
\frac{L^d}{d!}\int_0^L |g_{j,x}^{(d+1)}(s)|\,ds.
\]
Squaring and integrating over $0\le t\le L$, we obtain
$$
\|R_{j,x}\|_{L^2(0,L)}^2
\le
\frac{L^{2d+2}}{(d!)^2}
\int_0^L |g_{j,x}^{(d+1)}(s)|^2\,ds.
$$
Integrating this inequality over $x\in \mathbb T^2$, we get
$$
\int_{\mathbb T^2}\int_0^L |R_{j,x}(t)|^2\,dt\,dx
\le
\frac{L^{2d+2}}{(d!)^2}
\int_0^L \|D_j^{d+1}h_j\|_{L^2(\mathbb T^2)}^2\,dt.
$$
Using the Bernstein inequality from above,
\[
\int_{\mathbb T^2}\int_0^L |R_{j,x}(t)|^2\,dt\,dx
\le
L\left(\frac{(2\pi W L)^{d+1}}{d!}\right)^2
\|h_j\|_{L^2(\mathbb T^2)}^2.
\]
Set $$\ep_d := \left(\frac{(2\pi W L)^{d+1}}{d!}\right)^2.$$
Then,
\begin{equation}
    \label{two-strip-eq2}
    \frac1L\int_{\mathbb T^2}\int_0^L |R_{j,x}(t)|^2\,dt\,dx
\le \ep_d\, \|h_j\|_{L^2(\mathbb T^2)}^2.
\end{equation}

Now define the recombined Taylor approximation $$T_x(t)
:= e^{2\pi i a_1 t}\, T_{1,x}(t)
+ e^{2\pi i a_2 t}\, T_{2,x}(t),$$
and the recombined remainder
$$ R_x(t) := e^{2\pi i a_1 t}R_{1,x}(t)
+ e^{2\pi i a_2 t}R_{2,x}(t).$$
Thus,
\[
H_x(t)=T_x(t)+R_x(t).
\]
Using $|a+b|^2\le 2|a|^2+2|b|^2$ together with (\ref{two-strip-eq1}) and (\ref{two-strip-eq2}), we have
\begin{equation}
    \label{two-strip-eq3}
    \frac1L\int_{\mathbb T^2}\int_0^L |R_x(t)|^2\,dt\,dx
\le 2\ep_d \,\|h\|_{L^2(\mathbb T^2)}^2.
\end{equation}

We now use Nazarov's Tur\'an inequality. For each $x\in \mathbb T^2$, define
$$ E_{x, \ep}:=\{t\in [0,L]:\gamma_\ep(x+t\nu)\ge \eta\}.$$
We have $|E_{x, \ep}|\ge \theta L$ for every $x\in \mathbb T^2$ by assumption. We apply Lemma \ref{nazarov-Tur\'an-kovrizhkin} to
$$ T_x(t) = e^{2\pi i a_1 t}\, T_{1,x}(t)
+ e^{2\pi i a_2 t}\, T_{2,x}(t),$$
to get
$$
\|T_x\|_{L^2(0,L)}
\le
\left(\frac{CL}{|E_{x, \ep}|}\right)^{2(d+1)-\frac12}
\|T_x\|_{L^2(E_{x, \ep})}.
$$
Since $|E_{x, \ep}|\ge \theta L$, this gives
\[
\|T_x\|_{L^2(0,L)}
\le
\left(\frac{C}{\theta}\right)^{2d+\frac32}
\|T_x\|_{L^2(E_{x, \ep})}.
\]
Squaring this estimate, we obtain
$$ \int_0^L |T_x(t)|^2\,dt
\le A_{d,\theta}
\int_{E_{x, \ep}} |T_x(t)|^2\,dt,$$
where
$$A_{d,\theta} := \left(\frac{C}{\theta}\right)^{4d+3}.$$
Integrating over $x\in\mathbb T^2$, we get
\begin{equation}
    \label{two-strip-eq4}
    \int_{\mathbb T^2}\int_0^L |T_x(t)|^2\,dt\,dx
    \le A_{d,\theta} \int_{\mathbb T^2}\int_{E_{x, \ep}} |T_x(t)|^2\,dt\,dx.
\end{equation}

We now estimate $ \|h\|_{L^2(\mathbb T^2)}$ . We have
$$ \|h\|_{L^2(\mathbb T^2)}^2
= \frac1L \int_{\mathbb T^2}\int_0^L |h(x+t\nu)|^2\,dt\,dx
= \frac1L \int_{\mathbb T^2}\int_0^L |H_x(t)|^2\,dt\,dx.$$

Since $H_x=T_x+R_x$, we have
$$ |H_x(t)|^2
\le 2|T_x(t)|^2+2|R_x(t)|^2,$$
giving
$$\|h\|_{L^2(\mathbb T^2)}^2
\le \frac{2}{L} \int_{\mathbb T^2}\int_0^L |T_x(t)|^2\,dt\,dx
+ \frac{2}{L} \int_{\mathbb T^2}\int_0^L |R_x(t)|^2\,dt\,dx.$$

Using (\ref{two-strip-eq3}) and (\ref{two-strip-eq4}),
we get
$$ \|h\|_{L^2(\mathbb T^2)}^2
\le \frac{2A_{d,\theta} }{L}
\int_{\mathbb T^2}\int_{E_{x, \ep}} |T_x(t)|^2\,dt\,dx
+ 4\ep_d\|h\|_{L^2(\mathbb T^2)}^2.$$

Next, $T_x=H_x-R_x$ gives
$$|T_x(t)|^2
\le 2|H_x(t)|^2+2|R_x(t)|^2,$$
and so we have
$$\|h\|_{L^2(\mathbb T^2)}^2
\le \frac{4A_{d, \theta}}{L}
\int_{\mathbb T^2}\int_{E_{x, \ep}} |H_x(t)|^2\,dt\,dx
+ \frac{4A_{d, \theta}}{L}
\int_{\mathbb T^2}\int_{E_{x, \ep}} |R_x(t)|^2\,dt\,dx
+ 4\ep_d\, \|h\|_{L^2(\mathbb T^2)}^2.$$

Since $E_{x, \ep}\subset [0,L]$, we can estimate the second term by the full remainder term and use \eqref{two-strip-eq3} to get
$$ \|h\|_{L^2(\mathbb T^2)}^2
\le \frac{4A_{d, \theta}}{L}
\int_{\mathbb T^2}\int_{E_{x, \ep}} |H_x(t)|^2\,dt\,dx
+ (8A_{d, \theta}+4)\, \ep_d\, 
\|h\|_{L^2(\mathbb T^2)}^2.$$

We now choose $d$ sufficiently large so that $(8A_{d, \theta}+4)\,\ep_d \le \frac12.$  Absorbing the last term into
the left-hand side gives 
$$ \|h\|_{L^2(\mathbb T^2)}^2
\le \frac{8A_{d, \theta}}{L}
\int_{\mathbb T^2}\int_{E_{x, \ep}} |H_x(t)|^2\,dt\,dx.$$

Finally, by the definition of $E_{x, \ep}$, we have $\mathbf 1_{E_{x, \ep}}(t)\le \eta^{-1}\,\gamma_\ep(x+t\nu).$ Thus,
$$ \|h\|_{L^2(\mathbb T^2)}^2
\le \frac{8A_{d,\theta}}{\eta}
\int_{\mathbb T^2}|h(x)|^2\gamma_\ep(x)\,dx.$$

\end{proof}

\section{Trigonometric Polynomials with Spectra in Annuli of Variable Width}

Let $0\le \alpha \le \frac{1}{2}$ and define
$$
A_{R,\alpha}:=\{\xi\in\mathbb R^2: R\le |\xi|\le R+R^{-\alpha}\}.
$$
We prove the following observability estimate for functions with spectra in $A_{R,\alpha}$. 

\begin{tcolorbox}[colback=SeaGreen!10, colframe=SeaGreen!80, boxrule=0.5mm, left=2mm, right=2mm, boxsep=1mm, arc=1mm.]
\begin{theorem}
\label{thm:annular-observability}
Let $0\le\alpha\le\frac12$, and set $s:=\frac12-\alpha.$ Suppose that $\gamma\in H^s(\mathbb T^2)\cap L^\infty(\mathbb T^2)$ is non-negative and satisfies the generalized geometric control condition (GGCC) in every direction: there exist $L,c_0>0$ such that
$$
\liminf_{\delta\to0}
\inf_{x\in\mathbb T^2} \inf_{\nu \in S^1}
\frac{1}{|\Gamma_{x,\nu,L,\delta}|}
\int_{\Gamma_{x,\nu,L,\delta}}\gamma(y)\,dy
\ge c_0.
$$
Then there exists a constant $C>0$, independent of $R$, such that
\begin{equation*}
\|f\|_{L^2(\mathbb T^2)}^2 \le C\int_{\mathbb T^2}|f(x)|^2\, \gamma(x)\,dx,
\end{equation*}
whenever
$$\operatorname{supp}\widehat f
\subset A_{R,\alpha}\cap\mathbb Z^2.$$
\end{theorem}
\end{tcolorbox}

\begin{remark}
The circle in Theorem \ref{thm:annular-observability} may be replaced with any smooth closed curve $\Gamma\subset\mathbb R^2$ with everywhere positive curvature. More precisely, the same conclusion holds with
$$A_{R,\alpha}(\Gamma) :=\{\xi\in\mathbb R^2:\operatorname{dist}(\xi,R\Gamma)\le R^{-\alpha}\},$$
with constants depending on $\Gamma$.
\end{remark}

The key technical estimate used in the proof of Theorem \ref{thm:annular-observability} is a lattice point count. We record this separately as a lemma.

\begin{lemma}
    \label{lemma:lattice-point-count}
    Let $0\le \alpha\le \frac{1}{2}$ and fix $0<\ep<1$. For every $\sigma\in\mathbb Z^2$, we have the estimate
$$\#\left\{ \lambda\in A_{R,\alpha} \cap \mathbb Z^2: \lambda+\sigma\in A_{R,\alpha} \,\text{ and }\, |\sin\angle(\lambda,\lambda+\sigma)| \gtrsim \ep R^{-1/2} \right\}\lesssim \ep^{-1}R^{1/2-\alpha},$$
for all sufficiently large $R$, with constants in $\lesssim$ independent of $R$ and $\sigma$.
\end{lemma}

\begin{figure}[ht]
\centering

\begin{tikzpicture}[scale=0.8, line cap=round, line join=round]

\def\Rout{3.05}  
\def\Rin{2.75}  
\colorlet{labelblue}{blue!70!black}

\coordinate (O) at (0,0);
\coordinate (C) at (4.2,0);     
\coordinate (P) at (2.1,2.0);

\begin{scope}[even odd rule]
  \clip (O) circle (\Rout) (O) circle (\Rin);
  \path[fill=gray!25] (C) circle (\Rout) (C) circle (\Rin);
\end{scope}

\draw[thick] (O) circle (\Rout);
\draw[thick] (O) circle (\Rin);

\draw[thick] (C) circle (\Rout);
\draw[thick] (C) circle (\Rin);

\fill (O) circle (2pt);
\fill (C) circle (2pt);

\node[fill=white, inner sep=1.3pt] at ($(O)+(-0.30,-0.32)$) {$0$};
\node[fill=white, inner sep=1.3pt] at ($(C)+(0.46,-0.32)$) {$-\sigma$};

\draw[thick] (C) -- (O);
\node[fill=white, inner sep=1.3pt] at ($(O)!0.5!(C)+(0,-0.38)$) {$\sigma$};

\fill (P) circle (2pt);
\node[fill=white, inner sep=1.3pt] at ($(P)+(0,0.62)$) {$\lambda$};

\draw[thick] (O) -- (P);
\draw[thick] (C) -- (P);

\node[text=labelblue, fill=white, inner sep=1.3pt]
  at ($(O)!0.52!(P)+(-0.46,0.28)$) {$\lambda$};

\node[text=labelblue, fill=white, inner sep=1.3pt]
  at ($(C)!0.52!(P)+(0.62,0.28)$) {$\lambda+\sigma$};

\draw (P) ++(223.6:0.52)
  arc[start angle=223.6, end angle=316.4, radius=0.52];

\node[fill=white, inner sep=1.3pt] at ($(P)+(0,-0.82)$) {$\theta$};

\coordinate (Aout) at ($(O)+(142:\Rout)$);
\coordinate (Ain)  at ($(O)+(142:\Rin)$);

\draw[<->] (Ain) -- (Aout);

\node[text=labelblue, fill=white, inner sep=1.3pt]
  at ($(Ain)!0.5!(Aout)+(-0.62,0.34)$) {$h$};

\coordinate (Bout) at ($(C)+(38:\Rout)$);
\coordinate (Bin)  at ($(C)+(38:\Rin)$);

\draw[<->] (Bin) -- (Bout);

\node[text=labelblue, fill=white, inner sep=1.3pt]
  at ($(Bin)!0.5!(Bout)+(0.62,0.34)$) {$h$};

\end{tikzpicture}

\caption{}
\label{fig:lattice-count-1}
\end{figure}

The proof of the lemma will be given later; for now, we explain the geometric intuition behind the estimate. For a fixed $\sigma\in\mathbb Z^2$, we consider $\lambda\in A_{R,\alpha}\cap (A_{R,\alpha}-\sigma)$, i.e., points in the overlap of two annuli of radii $\approx R$ and radial thickness $h=R^{-\alpha}$, one centered at $0$ and the other at $-\sigma$. See Figure \ref{fig:lattice-count-1}. At such a point of intersection, the normals to the two circles are $\lambda$ and $\lambda+\sigma$, and we denote their angle of intersection by
$\theta :=\angle(\lambda,\lambda+\sigma)$. If the intersection is transverse, then the overlap of the two annuli along the circle has arclength on the order of $\frac{h}{\sin\theta}.$ In our application, the poker-chip decomposition forces the transversality condition $|\sin\theta|\gtrsim \ep R^{-1/2}$ and therefore the overlap length is  $\lesssim \ep^{-1}R^{1/2-\alpha}$. Lemma \ref{lemma:lattice-point-count} provides a discrete version of this heuristic.

\begin{proof}[Proof of Theorem \ref{thm:annular-observability}] 
Let $h:=R^{-\alpha}$ for the radial thickness parameter. We first prove the theorem for all sufficiently large $R$. The remaining bounded range of $R$ will be handled at the end by a compactness argument.

We divide the frequency annulus $A_{R,\alpha}$ into symmetric sector-pairs, see Figure \ref{fig:poker-chip}(a). The relevant angular scale is $\Delta\sim R^{-1/2}$. Indeed, on a sector of opening angle $\Delta$, the deviation of the circle of radius $R$ from its tangent line is $O(R\Delta^2)$. Since the annulus has radial thickness $h=R^{-\alpha}\le 1$, choosing $\Delta\sim R^{-1/2}$ ensures that each such sector is contained in a Euclidean strip of width $O(1)$. Thus, a symmetric sector-pair is contained in the union of two parallel strips of uniformly bounded width, see Figure \ref{fig:poker-chip}(b). This allows us to apply Theorem \ref{two-strip-theorem} to each symmetric sector-pair.

\begin{figure}[ht]
\centering
\begin{tikzpicture}[scale=0.9, line cap=round, line join=round]


\def\Rin{2.15}
\def\Rout{2.85}
\pgfmathsetmacro{\Rmid}{0.5*(\Rin+\Rout)}

\def\N{10}
\pgfmathsetmacro{\DeltaAng}{180/\N}
\pgfmathsetmacro{\phiang}{8}
\pgfmathsetmacro{\deltaratio}{0.22}
\pgfmathsetmacro{\deltaAng}{\deltaratio*\DeltaAng}
\pgfmathtruncatemacro{\TwoNminusOne}{2*\N - 1}

\colorlet{deletedfill}{red!32}
\colorlet{keptfill}{blue!15}
\colorlet{keptedge}{blue!65!black}
\colorlet{stripfill}{blue!8}
\colorlet{stripedge}{blue!70!black}

\begin{scope}[shift={(0,0)}]
\node[font=\bfseries] at (-3.25,3.15) {(a)};
  \foreach \m in {0,...,\TwoNminusOne} {
    \pgfmathsetmacro{\a}{\phiang + \m*\DeltaAng}
    \pgfmathsetmacro{\b}{\a + \deltaAng}
    \pgfmathsetmacro{\c}{\a + \DeltaAng}

    \path[fill=deletedfill, draw=none]
      (\a:\Rin) -- (\a:\Rout)
      arc[start angle=\a, end angle=\b, radius=\Rout]
      -- (\b:\Rin)
      arc[start angle=\b, end angle=\a, radius=\Rin]
      -- cycle;

    \path[fill=keptfill, draw=none]
      (\b:\Rin) -- (\b:\Rout)
      arc[start angle=\b, end angle=\c, radius=\Rout]
      -- (\c:\Rin)
      arc[start angle=\c, end angle=\b, radius=\Rin]
      -- cycle;
  }

  \draw[thick] (0,0) circle (\Rout);
  \draw[thick] (0,0) circle (\Rin);

  \fill (0,0) circle (1.2pt);

  
  \pgfmathsetmacro{\angThickStart}{\phiang + 2*\DeltaAng + \deltaAng}
  \pgfmathsetmacro{\angThickEnd}{\phiang + 2*\DeltaAng + \DeltaAng}
  
  \draw[decorate, decoration={brace, mirror, amplitude=5pt, raise=3pt}, thick] 
    (\angThickStart:\Rout) arc[start angle=\angThickStart, end angle=\angThickEnd, radius=\Rout];
  \node[font=\small] at ({0.5*(\angThickStart+\angThickEnd+ 5)}:\Rout+0.7) {$\sim \hspace{-0.71mm}\sqrt{R}$};

  \pgfmathsetmacro{\angThinStart}{\phiang + 4*\DeltaAng}
  \pgfmathsetmacro{\angThinEnd}{\phiang + 4*\DeltaAng + \deltaAng}
  
  \draw[dotted, thick] (\angThinStart:\Rout + 0.05) -- (\angThinStart:\Rout + 0.75);
  \draw[dotted, thick] (\angThinEnd:\Rout + 0.05) -- (\angThinEnd:\Rout + 0.75);
  
  \draw[<-, >=stealth, thick] (\angThinStart:\Rout + 0.5) arc[start angle=\angThinStart, end angle=\angThinStart-6, radius=\Rout + 0.5];
  \draw[<-, >=stealth, thick] (\angThinEnd:\Rout + 0.5) arc[start angle=\angThinEnd, end angle=\angThinEnd+6, radius=\Rout + 0.5];
  
  \node[font=\small] at ({0.5*(\angThinStart+\angThinEnd)}:\Rout + 1.1) {$\epsilon\sqrt{R}$};

\end{scope}

\begin{scope}[shift={(8.4,0)}]
\node[font=\bfseries] at (-3.25,3.15) {(b)};
  \pgfmathsetmacro{\thetac}{32}
  \pgfmathsetmacro{\sectorW}{0.78*\DeltaAng}
  \pgfmathsetmacro{\aone}{\thetac-0.5*\sectorW}
  \pgfmathsetmacro{\bone}{\thetac+0.5*\sectorW}
  \pgfmathsetmacro{\atwo}{\aone+180}
  \pgfmathsetmacro{\btwo}{\bone+180}
  \pgfmathsetmacro{\stripangle}{\thetac+90}

  \begin{scope}[shift={({\thetac}:\Rmid)}, rotate=\stripangle]
    \draw[fill=stripfill, draw=stripedge, line width=0.9pt]
      (-2.05,-0.36) rectangle (2.05,0.36);
      
    \draw[dotted, thick] (2.05, -0.36) -- (2.35, -0.36);
    \draw[dotted, thick] (2.05, 0.36) -- (2.35, 0.36);
    \draw[<->, >=stealth, thick] (2.2, -0.36) -- (2.2, 0.36);
    \node[font=\small] at (2.7, 0) {$O(1)$};

    \node[font=\small] at (0, -0.75) {$\Omega_i$};
  \end{scope}

  \begin{scope}[shift={({\thetac+180}:\Rmid)}, rotate=\stripangle]
    \draw[fill=stripfill, draw=stripedge, line width=0.9pt]
      (-2.05,-0.36) rectangle (2.05,0.36);

    \node[font=\small] at (0, 0.75) {$\Omega_i$};
  \end{scope}

  \path[fill=keptfill, draw=keptedge, line width=0.9pt]
    (\aone:\Rin) -- (\aone:\Rout)
    arc[start angle=\aone, end angle=\bone, radius=\Rout]
    -- (\bone:\Rin)
    arc[start angle=\bone, end angle=\aone, radius=\Rin]
    -- cycle;

  \path[fill=keptfill, draw=keptedge, line width=0.9pt]
    (\atwo:\Rin) -- (\atwo:\Rout)
    arc[start angle=\atwo, end angle=\btwo, radius=\Rout]
    -- (\btwo:\Rin)
    arc[start angle=\btwo, end angle=\atwo, radius=\Rin]
    -- cycle;

  \draw[thick] (0,0) circle (\Rout);
  \draw[thick] (0,0) circle (\Rin);

  \fill (0,0) circle (1.2pt);

\end{scope}

\end{tikzpicture}
\caption{(a) The poker-chip decomposition of the annulus $A_{R,\alpha}$ into thin and thick strips of width about $\ep\sqrt R$ and $\sqrt{R}$ respectively. (b) $\Omega_i$ is covered by two parallel strips of bounded width.}
\label{fig:poker-chip}
\end{figure}
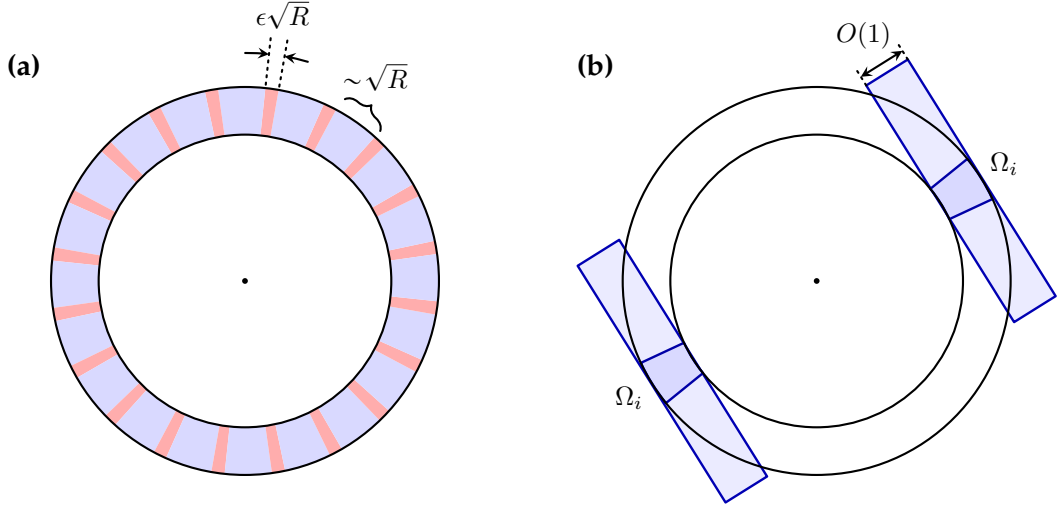

This poker-chip decomposition is made precise now. We work with angles modulo $\pi$, so that each angular interval represents both a sector and its reflected sector. The thin sectors have width about $\ep\sqrt R$ along the annulus, while the thick sectors have width approximately $\sqrt{R}$ along the annulus.

Choose an integer $N\sim \sqrt R$ and set $\Delta:=\frac{\pi}{N}.$ Thus $\Delta\sim R^{-1/2}.$ Let
$\delta:=\ep\Delta\sim \ep R^{-1/2}.$
For $\phi\in[0,\Delta)$, define
$$
S_\phi
:=
\bigcup_{k=0}^{N-1}
[k\Delta+\phi,\ k\Delta+\phi+\delta),
$$
where the union is understood modulo $\pi$. Equivalently, on the full circle $\mathbb R/2\pi\mathbb Z$, each interval is taken together with its reflected copy shifted by $\pi$. Define the corresponding frequency set
$$
E_\phi
:=
\{\xi\in A_{R,\alpha}\cap\mathbb Z^2:
\arg \xi \bmod \pi \in S_\phi\}.
$$
We decompose
$$
f=f_1+f_2
$$
by setting
$\widehat{f_1}(\xi)
:= \widehat f(\xi)\mathbf 1_{E_\phi}(\xi)$, and $f_2:=f-f_1$. A standard pigeonholing argument lets us choose $\phi \in [0, \Delta)$ such that $f_1$ carries negligible $L^2$ mass. Then, we analyze $f_2$. In particular, first observe that
$$
\|f_1\|_2^2
= \sum_{\xi\in\mathbb Z^2}
|\widehat{f_1}(\xi)|^2 
= \sum_{\xi\in A_{R,\alpha}\cap\mathbb Z^2}
|\widehat f(\xi)|^2\, \mathbf 1_{E_\phi}(\xi)
= \sum_{\xi\in A_{R,\alpha}\cap\mathbb Z^2}
|\widehat f(\xi)|^2\, \mathbf 1_{S_\phi}(\arg \xi\bmod \pi).$$

For each fixed $\theta\in \mathbb R/\pi\mathbb Z$, we have $\frac1\Delta \int_0^\Delta
\mathbf 1_{S_\phi}(\theta)\,d\phi = \frac{\delta}{\Delta} = \ep,$ so averaging the above expression over $\phi \in [0, \Delta)$ gives
$$
\frac1\Delta \int_0^\Delta \|f_1\|_2^2\, d\phi
=
\ep
\sum_{\xi\in A_{R,\alpha}\cap\mathbb Z^2}
|\widehat f(\xi)|^2
=
\ep\,\|f\|_2^2.
$$
Thus, we can find some $\phi\in[0,\Delta)$ such that
\begin{equation}
    \label{eq:pigeonhole-small}
    \|f_1\|_2^2 \le \ep\|f\|_2^2.
\end{equation}
We fix such a $\phi$ from now on. The complement of $S_\phi$ in $\mathbb R/\pi\mathbb Z$ is a union of angular intervals of length comparable to $\Delta$, separated from each other by gaps of angular size $\delta=\ep\Delta$. Denote these intervals by $I_i$, where $1\le i\le N$. Define the corresponding symmetric sector-pairs
$$
\Omega_i
:=
\{\xi\in A_{R,\alpha}: \arg \xi \bmod \pi \in I_i\}.
$$
Equivalently, if $I_i$ is represented as an interval in $\mathbb R/2\pi\mathbb Z$, then $\Omega_i$ is the union of the strip over $I_i$ and its reflected strip over $I_i+\pi$. Set
$\widehat{f_{2,i}}(\xi)
:= \widehat{f_2}(\xi)\mathbf 1_{\Omega_i}(\xi).$ Then,
$$
f_2=\sum_{i=1}^N f_{2,i},
$$
and the Fourier supports of the functions $f_{2,i}$ are disjoint. By the observability inequality of Theorem \ref{two-strip-theorem}, we have
\begin{equation}
    \label{eq:f_2-disjoint}
    \|f_2\|_2^2 = \sum_{i=1}^N \|f_{2,i}\|_2^2 \le C_{\mathrm{str}} \sum_{i=1}^N \|f_{2,i} \sqrt{\gamma}\|_2^2.
\end{equation}
It remains to show
$$
\sum_{i=1}^N \|f_{2,i}\sqrt{\gamma}\|_2^2
\lesssim
\left\|\sum_{i=1}^N f_{2,i}\sqrt{\gamma}\right\|_2^2.
$$
This is where the Sobolev regularity of $\gamma$ plays a role. Note that if $\lambda\in \Omega_i$ and $\mu\in \Omega_j$ for $i \ne j$, and $\theta=\angle(\lambda,\mu)\in[0,\pi]$, then $$\min(\theta,\pi-\theta)
\gtrsim \ep R^{-1/2}.$$

Set $t:=\min(\theta,\pi-\theta)$. Since $t\in[0,\pi/2]$ and
$\sin t\gtrsim t$, we have
$$
|\sin\theta|
=
\sin(\min(\theta,\pi-\theta))
=
\sin t
\gtrsim t
\gtrsim \ep R^{-1/2}.
$$
Now, we use the cross product expressions $ |\lambda\times\mu|
= |\lambda|\,|\mu|\,|\sin\theta|$ and $\lambda\times\mu
= \lambda\times(\mu-\lambda)$. We get $$ |\lambda|\,|\mu|\,|\sin\theta|
= |\lambda\times(\mu-\lambda)| \le |\lambda|\,|\mu-\lambda|,$$
which simplifies to $$|\mu-\lambda|
\ge |\mu|\,|\sin\theta| \gtrsim R\cdot \ep R^{-1/2}
= \ep\sqrt R.$$

We have
$$\|f_2\sqrt{\gamma}\|_2^2
= \int_{\mathbb T^2} \left|\sum_i f_{2,i}(x)\right|^2
\gamma(x)\,dx = \sum_i \int_{\mathbb T^2} |f_{2,i}(x)|^2\gamma(x)\,dx
+ \sum_{i\neq j} \int_{\mathbb T^2} f_{2,i}(x)\overline{f_{2,j}(x)}\gamma(x)\,dx,$$
so that
\begin{equation}
    \label{eq:off-diag-small}
    \sum_i \|f_{2,i}\sqrt{\gamma}\|_2^2
= \|f_2\sqrt{\gamma}\|_2^2
- \sum_{i\neq j}
\int_{\mathbb T^2}
f_{2,i}(x)\overline{f_{2,j}(x)}\gamma(x)\,dx.
\end{equation}
We will show that the off-diagonal contribution is small. Using Fourier series, write
$$
f_{2,i}(x)
=
\sum_{\lambda\in \Omega_i\cap\mathbb Z^2}
\widehat{f_{2,i}}(\lambda)\, e^{2\pi i\lambda\cdot x}.
$$
Then
$$
\int_{\mathbb T^2}
f_{2,i}(x)\overline{f_{2,j}(x)}\gamma(x)\,dx
=
\sum_{\lambda\in \Omega_i}
\sum_{\mu\in \Omega_j}
\widehat{f_{2,i}}(\lambda)
\overline{\widehat{f_{2,j}}(\mu)}
\widehat\gamma(\mu-\lambda).
$$
Therefore, using the separation estimate $|\lambda-\mu|\gtrsim \ep\sqrt R$ for $i\neq j$,
$$
\left|
\sum_{i\neq j}
\int_{\mathbb T^2}
f_{2,i}\overline{f_{2,j}}\gamma
\right|
\le
\sum_{|\sigma|\gtrsim \ep\sqrt R}
|\widehat\gamma(\sigma)|
\sum_{\substack{i\neq j,\; \lambda\in \Omega_i,\; \mu\in \Omega_j \\ \mu-\lambda=\sigma}}
|\widehat{f_{2,i}}(\lambda)|
|\widehat{f_{2,j}}(\mu)|.
$$
For each $\sigma\in\mathbb Z^2$, define
$$
B(\sigma)
:=
\sum_{\substack{i\neq j,\ \lambda\in \Omega_i,\ \mu\in \Omega_j\\ \mu-\lambda=\sigma}}
|\widehat{f_{2,i}}(\lambda)|
|\widehat{f_{2,j}}(\mu)|.
$$
Then
$$
\left|
\sum_{i\neq j}
\int_{\mathbb T^2}
f_{2,i}\overline{f_{2,j}}\gamma
\right|
\le
\sum_{|\sigma|\gtrsim \ep\sqrt R}
|\widehat\gamma(\sigma)|B(\sigma).
$$
Applying Cauchy--Schwarz in $\sigma$, we get
$$
\left|
\sum_{i\neq j}
\int_{\mathbb T^2}
f_{2,i}\overline{f_{2,j}}\gamma
\right|
\le
\left(
\sum_{|\sigma|\gtrsim \ep\sqrt R}
|\widehat\gamma(\sigma)|^2
\right)^{1/2}
\left(
\sum_{|\sigma|\gtrsim \ep\sqrt R}
B(\sigma)^2
\right)^{1/2}.
$$

We now estimate the second factor. For fixed $\sigma\in\mathbb Z^2$, define
$$
N(\sigma)
:=
\#\left\{
(i,j,\lambda,\mu):
i\neq j,\ 
\lambda\in \Omega_i\cap\mathbb Z^2,\ 
\mu\in \Omega_j\cap\mathbb Z^2,\ 
\mu-\lambda=\sigma
\right\}.
$$
By the Cauchy--Schwarz inequality,
$$
B(\sigma)^2
\le
N(\sigma)
\sum_{\substack{i\neq j,\ \lambda\in \Omega_i,\ \mu\in \Omega_j\\ \mu-\lambda=\sigma}}
|\widehat{f_{2,i}}(\lambda)|^2
|\widehat{f_{2,j}}(\mu)|^2.
$$

Since $\mu=\lambda+\sigma$ is determined by $\lambda$, and since we are only counting pairs coming from distinct strip-pairs, we have $N(\sigma)
\le \#\Lambda_\sigma,$
where
$$\Lambda_\sigma := \left\{ \lambda\in A_{R,\alpha} \cap \mathbb Z^2: \lambda+\sigma\in A_{R,\alpha}  \,\text{ and }\, |\sin\angle(\lambda,\lambda+\sigma)|\gtrsim \ep R^{-1/2} \right\}.$$
By Lemma \ref{lemma:lattice-point-count}, $\#\Lambda_\sigma \lesssim \ep^{-1}R^{1/2-\alpha}$. This leads to 
$$
B(\sigma)^2 \lesssim \ep^{-1}R^{1/2-\alpha}
\sum_{\substack{i\neq j,\ \lambda\in \Omega_i,\ \mu\in \Omega_j\\ \mu-\lambda=\sigma}}
|\widehat{f_{2,i}}(\lambda)|^2
|\widehat{f_{2,j}}(\mu)|^2.
$$
Summing in $\sigma$, we get
$$
\sum_{|\sigma|\gtrsim \ep\sqrt R}
B(\sigma)^2
\lesssim
\ep^{-1}R^{1/2-\alpha}
\sum_{|\sigma|\gtrsim \ep\sqrt R}
\sum_{\substack{i\neq j,\ \lambda\in \Omega_i,\ \mu\in \Omega_j\\ \mu-\lambda=\sigma}}
|\widehat{f_{2,i}}(\lambda)|^2
|\widehat{f_{2,j}}(\mu)|^2.
$$
Clearly,
$$
\sum_{|\sigma|\gtrsim \ep\sqrt R} \sum_{\substack{i\neq j,\ \lambda\in \Omega_i,\ \mu\in \Omega_j\\ \mu-\lambda=\sigma}}
|\widehat{f_{2,i}}(\lambda)|^2 |\widehat{f_{2,j}}(\mu)|^2
\le \sum_{i,\lambda} |\widehat{f_{2,i}}(\lambda)|^2
\sum_{j,\mu} |\widehat{f_{2,j}}(\mu)|^2 = \|f_2\|_2^4.
$$
Therefore,
$$\sum_{|\sigma|\gtrsim \ep\sqrt R}
B(\sigma)^2 \lesssim
\ep^{-1}R^{1/2-\alpha}\, \|f_2\|_2^4.
$$
Substituting this into the previous Cauchy--Schwarz estimate for the off-diagonal term, we obtain
$$
\left|
\sum_{i\neq j}
\int_{\mathbb T^2}
f_{2,i}\overline{f_{2,j}}\gamma
\right|
\lesssim
\ep^{-1/2}R^{1/4-\alpha/2}
\left(
\sum_{|\sigma|\gtrsim \ep\sqrt R}
|\widehat\gamma(\sigma)|^2
\right)^{1/2}
\|f_2\|_2^2.
$$

We now use Sobolev regularity. If $\gamma\in H^s(\mathbb T^2)$, then
$$
\left( \sum_{|\sigma|\gtrsim \ep\sqrt R}
|\widehat\gamma(\sigma)|^2 \right)^{1/2}
\lesssim \ep^{-s}R^{-s/2}
\left(\sum_{|\sigma|\gtrsim \ep\sqrt R}
|\widehat\gamma(\sigma)|^2 |\sigma|^{2s}
\right)^{1/2},$$
giving
$$
\left| \sum_{i\neq j} \int_{\mathbb T^2} f_{2,i}\overline{f_{2,j}}\gamma
\right| \lesssim \ep^{-s-1/2} R^{1/4-\alpha/2-s/2} \left(\sum_{|\sigma|\gtrsim \ep\sqrt R}
|\widehat\gamma(\sigma)|^2 |\sigma|^{2s} \right)^{1/2} \|f_2\|_2^2.
$$
If $s>\frac12-\alpha$, then the power of $R$ already decays. At the endpoint
$s=\frac12-\alpha$, the power of $R$ vanishes, but the Sobolev tail still tends to zero as $R\to\infty$, since $\gamma\in H^s(\mathbb T^2)$. Thus the off-diagonal term can be made arbitrarily small for all sufficiently large $R$. In particular, 
$$
\left| \sum_{i\neq j} \int_{\mathbb T^2} f_{2,i}\overline{f_{2,j}}\gamma
\right| \le \frac{1}{2C_{\mathrm{str}}} \|f_2\|_2^2,
$$
for all $R \ge R_0$ for some $R_0 > 0$. Therefore, $$\sum_i \|f_{2,i}\sqrt{\gamma}\|_2^2 \le \|f_2\sqrt{\gamma}\|_2^2 + \frac{1}{2C_{\mathrm{str}}} \|f_2\|_2^2,$$
using (\ref{eq:off-diag-small}). Combining this with (\ref{eq:f_2-disjoint}), we get
$$\|f_2\|_2^2 \le 2C_{\mathrm{str}}\, \|f_2\sqrt{\gamma}\|_2^2,$$
and writing $f_2 = f - f_1$, we can estimate the right side as 
$$\|f_2\sqrt{\gamma}\|_2^2 \le 2\|f\sqrt{\gamma}\|_2^2 + 2\|f_1\sqrt{\gamma}\|_2^2 \le 2\|f\sqrt{\gamma}\|_2^2 + 2\ep \|\gamma\|_\infty\|f\|_2^2,$$
by (\ref{eq:pigeonhole-small}). Finally,
$$\|f\|_2^2 = \|f_1\|_2^2 + \|f_2\|_2^2 \le \ep \|f\|_2^2 + 4C_{\mathrm{str}}\|f\sqrt{\gamma}\|_2^2
+ 4C_{\mathrm{str}}\,\ep\|\gamma\|_\infty\|f\|_2^2,$$
and choose $\ep > 0$ small enough to conclude $$\|f\|_2 \lesssim \|f\sqrt{\gamma}\|_2,$$
for every $f \in L^2(\T^2)$ satisfying $\operatorname{supp}\widehat f\subset A_{R,\alpha}\cap\mathbb Z^2$.

We now handle the finite piece $1\le R\le R_0$. In this range, the relevant functions lie in a fixed finite-dimensional subspace $\mathcal V \subset L^2(\mathbb T^2)$. Suppose the estimate failed on this bounded range. Then there would exist $f_n\in \mathcal{V}$ such that $\|f_n\|_2=1$ for all $n \ge 1$ but $\|f_n\sqrt{\gamma}\|_2\to 0.$
After passing to a subsequence, $f_n\to f$ in $L^2$, with $\|f\|_2=1$. Since $\gamma\in L^\infty$, we also have $f_n\sqrt{\gamma}\to f\sqrt{\gamma}$ in $L^2$. Hence $f\sqrt{\gamma}=0$, so $f=0$ a.e. on $\{\gamma>0\}$.

GGCC implies that $\{\gamma>0\}$ has positive measure. Therefore, the trigonometric polynomial $f$ vanishes on a set of positive measure, forcing $f\equiv0$, which contradicts $\|f\|_2=1$. This proves the estimate for $1\le R\le R_0$ and completes the proof.

\end{proof}

\section{A Lattice Count for Transversely Overlapping Annuli}

We now prove a lattice point count for transversely overlapping annuli, as described in Lemma \ref{lemma:lattice-point-count}. The statement is repeated below for convenience.
\begin{lemma*}
    Let $0\le \alpha\le \frac{1}{2}$ and fix $0<\ep<1$. For every $\sigma\in\mathbb Z^2$, we have the estimate
$$\#\left\{ \lambda\in A_{R,\alpha} \cap \mathbb Z^2: \lambda+\sigma\in A_{R,\alpha} \,\text{ and }\, |\sin\angle(\lambda,\lambda+\sigma)| \gtrsim \ep R^{-1/2} \right\}\lesssim \ep^{-1}R^{1/2-\alpha},$$
for all sufficiently large $R$, with constants in $\lesssim$ independent of $R$ and $\sigma$.
\end{lemma*}

\begin{proof}[Proof of Lemma \ref{lemma:lattice-point-count}]
    Let $h:=R^{-\alpha}$ and $$ E_\sigma := \left\{x\in A_{R,\alpha}: x+\sigma\in A_{R,\alpha} \,\text{ and }\, |\sin\angle(x,x+\sigma)|\gtrsim \ep R^{-1/2} \right\}.$$
Then, we need to estimate $\# (E_\sigma\cap\mathbb Z^2)$.

\medskip
\noindent \textbf{Step 1: Projection to $RS^1$.} 

Let $\pi(x):=R\frac{x}{|x|}$ be the radial projection onto the circle $RS^1$, and consider the set $\Gamma_\sigma:=\pi(E_\sigma)$. We justify that the transversality condition $|\sin\angle(x,x+\sigma)|\gtrsim \ep R^{-1/2}$ is stable under this projection. Let $y :=\pi(x)$. Then,
$$ |x-y| = \left|x-R\frac{x}{|x|}\right|
= \bigl||x|-R\bigr| \le h,$$ 
since $x\in A_{R,\alpha}$. Set
$u:=x+\sigma$ and $v:=y+\sigma$. Then, $|u-v| \le h.$ Moreover, $x+\sigma\in A_{R,\alpha}$, so $|u|\sim R$, and by the triangle inequality, we have
$$ \bigl||v|-R\bigr| \le \bigl||v|-|u|\bigr|+\bigl||u|-R\bigr|
\le |v-u|+h \le  2h.$$
Hence also $|v|\approx R$. We now compare the directions of $u$ and $v$. We will show that
\begin{equation}
    \label{eq:proj-h/R}
    \left| \frac{u}{|u|} - \frac{v}{|v|} \right| \lesssim \frac{h}{R}.
\end{equation}

\begin{figure}[ht]
\centering
\begin{tikzpicture}[
    scale=1,
    line cap=round,
    line join=round,
    >=Latex,
    every node/.style={font=\small}
]

\def\R{3.15}       
\def\ang{57}       
\def\dr{0.46}      

\colorlet{projcol}{blue!70!black}
\colorlet{transcol}{green!55!black}
\colorlet{errcol}{red!75!black}

\coordinate (O)  at (0,0);
\coordinate (y)  at (\ang:\R);
\coordinate (x)  at (\ang:{\R+\dr});

\coordinate (sig) at (1.95,-0.34);

\coordinate (yp) at ($(y)+(sig)$);   
\coordinate (xp) at ($(x)+(sig)$);   

\coordinate (rayYp) at ($(O)!1.08!(yp)$);
\coordinate (rayXp) at ($(O)!1.08!(xp)$);

\draw[thick] (O) circle (\R);
\fill (O) circle (1.4pt);
\node[below left] at (O) {$0$};

\node[fill=white, inner sep=1pt] at ($(146:\R)+(-0.4,0.15)$) {$RS^1$};

\draw[thick] (O) -- (x);

\draw[->, thick] (x) -- (xp);
\draw[->, thick] (y) -- (yp);

\node[fill=white, inner sep=1pt] at ($(y)!0.57!(yp)+(-0.15,-0.18)$) {$\sigma$};

\draw[dashed, thick, transcol] (O) -- (rayYp);
\draw[dashed, thick, projcol]  (O) -- (rayXp);

\fill (x)  circle (1.6pt);
\fill (y)  circle (1.6pt);
\fill (xp) circle (1.6pt);
\fill (yp) circle (1.6pt);

\node[fill=white, inner sep=1pt] at ($(x)+(0.15,0.22)$) {$x$};
\node[fill=white, inner sep=1pt] at ($(y)+(-0.86,-0.15)$) {$y=\pi(x)$};

\node[fill=white, inner sep=1pt] at ($(xp)+(0.46,0.10)$) {$x+\sigma$};
\node[fill=white, inner sep=1pt] at ($(yp)+(0.56,-0.18)$) {$y+\sigma$};

\draw[<->, projcol, thin]
    ($(y)+(0.07,-0.03)$) -- ($(x)+(0.07,-0.03)$);
\node[text=projcol, fill=white, inner sep=1pt]
    at ($(y)!0.5!(x)+(0.35,-0.05)$) {$h$};

\draw[<->, projcol, thin]
    ($(yp)+(0.10,-0.03)$) -- ($(xp)+(0.10,-0.03)$);
\node[text=projcol, fill=white, inner sep=1pt]
    at ($(yp)!0.5!(xp)+(0.35,-0.05)$) {$h$};

\pgfmathanglebetweenpoints{\pgfpointanchor{O}{center}}{\pgfpointanchor{y}{center}}
\let\angbase\pgfmathresult

\pgfmathanglebetweenpoints{\pgfpointanchor{O}{center}}{\pgfpointanchor{yp}{center}}
\let\angyp\pgfmathresult

\pgfmathanglebetweenpoints{\pgfpointanchor{O}{center}}{\pgfpointanchor{xp}{center}}
\let\angxp\pgfmathresult

\draw[transcol, thick]
  ($(O)+(\angbase:0.92)$)
  arc[start angle=\angbase, end angle=\angyp, radius=0.92];

\node[text=transcol, fill=white, inner sep=0.8pt]
  at ($(O)+({0.5*(\angbase+\angyp)}:1.08)$) {$\theta_y$};

\draw[projcol, thick]
  ($(O)+(\angbase:1.38)$)
  arc[start angle=\angbase, end angle=\angxp, radius=1.38];

\node[text=projcol, fill=white, inner sep=0.8pt]
  at ($(O)+({0.5*(\angbase+\angxp)}:1.57)$) {$\theta_x$};

\draw[errcol, thick]
  ($(O)+(\angyp:1.95)$)
  arc[start angle=\angyp, end angle=\angxp, radius=1.95];

\node[text=errcol, fill=white, inner sep=0.8pt]
  at ($(O)+({0.5*(\angyp+\angxp)}:2.45)+(0.18,-0.5)$) {$O(h/R)$};

\end{tikzpicture}
\caption{}
\label{fig:projection-stability}
\end{figure}

Indeed,
$$ \frac{u}{|u|} - \frac{v}{|v|}
= \frac{u-v}{|u|} + v\left(\frac1{|u|}-\frac1{|v|}\right),
$$
and taking absolute values, we get
$$
\left| \frac{u}{|u|} - \frac{v}{|v|} \right| \le \frac{|u-v|}{|u|} +
|v| \frac{\bigl||v|-|u|\bigr|}{|u||v|}  \le \frac{|u-v|}{|u|} + \frac{|u-v|}{|u|} \lesssim \frac{|u-v|}{R},$$
by the reverse-triangle inequality. As $|u-v|\le h$, we obtain (\ref{eq:proj-h/R}). In other words, replacing $x$ by its radial projection $y$ changes the direction of $x+\sigma$ by at most $O(h/R)=O(R^{-1-\alpha})$; see Figure \ref{fig:projection-stability}. On the other hand, for fixed $\ep>0$ and sufficiently large $R$, we have $R^{-1-\alpha} \lesssim \ep R^{-1/2}.$ Thus, if $$|\sin\angle(x,x+\sigma)| \gtrsim \ep R^{-1/2},$$ then, with a different implicit constant, we have $$|\sin\angle(y,y+\sigma)|
\gtrsim \ep R^{-1/2}.$$

\medskip
\noindent \textbf{Step 2: How long is the projection?}

We now estimate $\mathcal H^1(\Gamma_\sigma)$. First, parametrize the circle $RS^1$ by arclength as follows. Define $$ y(s):= R\left(\cos\frac{s}{R},\sin\frac{s}{R}\right),$$
where $0\le s < 2\pi R$. Then, $|y'(s)|=1.$ Also define
$$ \Phi(s):=|y(s)+\sigma|,$$
and note that
$$ \Phi'(s) = \frac{y(s)+\sigma}{|y(s)+\sigma|}\cdot y'(s).$$
Since $y'(s) \perp y(s)$, we can rewrite the above expression as
$$\Phi'(s) = \sin \angle (y(s), y(s) + \sigma),$$ and we have 
$$
|\Phi'(s)|
=
|\sin\angle(y(s),y(s)+\sigma)|
\gtrsim
\ep R^{-1/2},
$$
for every $s$ such that $y(s)\in \Gamma_\sigma$. Moreover, $y(s)\in \Gamma_\sigma$ also implies $|\Phi(s)-R|\lesssim h.$ 

In order to estimate the length of $\Gamma_\sigma$, it will be useful to observe that $\Phi$ has only $O(1)$ monotonicity intervals on the circle. After a rotation, we may assume $\sigma=(d,0),$ where $d:=|\sigma|$. We have
$$ \Phi(s)^2 = |y(s)+\sigma|^2 = R^2+d^2+2Rd\cos\frac{s}{R}.$$
Set $F(s):=\Phi(s)^2$. Then,
$$ F'(s) = -2d\sin\frac{s}{R},$$
and $F'(s)=0$ only when $s = 0$ or $s = \pi R$.
Hence, $F=\Phi^2$ has only $O(1)$ critical points and therefore only $O(1)$ monotonicity intervals. Since on the set under consideration we have $\Phi(s)\sim R$, $\Phi$ is nonzero there, and so $\Phi=\sqrt F$ has the same monotonicity intervals as $F$.

\begin{figure}[ht]
\centering
\begin{tikzpicture}[
    scale=1.0,
    line cap=round,
    line join=round,
    >=Latex,
    every node/.style={font=\small}
]

\def\xmin{0}
\def\xmax{11}
\def\ymin{0}
\def\ymax{5.8}

\def\Rlow{2.35}
\def\Rmid{2.75}
\def\Rhigh{3.15}

\def\xJleft{1.55}
\def\xJright{9.20}

\def\xAleft{3.35}
\def\xAright{7.70}

\def\xEleft{4.55}
\def\xEright{6.55}

\colorlet{bandfill}{blue!8}
\colorlet{bandline}{blue!65!black}
\colorlet{Acol}{blue!75!black}
\colorlet{Ecol}{red!75!black}

\draw[->] (\xmin,\ymin) -- (\xmax,\ymin) node[below right] {$s$};
\draw[->] (\xmin,\ymin) -- (\xmin,\ymax) node[left] {$\Phi(s)$};

\fill[gray!10] (\xJleft,\ymin) rectangle (\xJright,\ymax-0.35);
\draw[dashed, gray!55] (\xJleft,\ymin) -- (\xJleft,\ymax-0.35);
\draw[dashed, gray!55] (\xJright,\ymin) -- (\xJright,\ymax-0.35);

\node[gray!60!black] at ({0.5*(\xJleft+\xJright)},5.25)
{$\Phi$ monotone on $J$};

\draw[decorate,decoration={brace,mirror,amplitude=4pt}, gray!70]
  (\xJleft,-0.12) -- (\xJright,-0.12)
  node[midway,below=6pt,black] {$J$};

\fill[bandfill] (\xmin,\Rlow) rectangle (\xmax,\Rhigh);

\draw[dashed, bandline] (\xmin,\Rhigh) -- (\xmax,\Rhigh);
\draw[dotted, bandline] (\xmin,\Rmid) -- (\xmax,\Rmid);
\draw[dashed, bandline] (\xmin,\Rlow) -- (\xmax,\Rlow);

\node[left, bandline] at (\xmin,\Rhigh) {$R+Ch$};
\node[left, bandline] at (\xmin,\Rmid) {$R$};
\node[left, bandline] at (\xmin,\Rlow) {$R-Ch$};

\draw[<->, bandline] (10.9,\Rlow) -- (10.9,\Rhigh);
\node[right, bandline] at (11.00,\Rmid) {$2Ch$};

\draw[very thick]
  (0.7,0.95)
  .. controls (1.8,1.15) and (2.65,1.75) ..
  (\xAleft,\Rlow)
  .. controls (4.25,2.62) and (5.10,2.72) ..
  (5.75,\Rmid)
  .. controls (6.55,2.88) and (7.10,3.00) ..
  (\xAright,\Rhigh)
  .. controls (8.35,3.65) and (9.10,4.65) ..
  (10.25,5.25);

\node[above right] at (9.35,4.95) {$\Phi$};

\draw[dashed, Acol] (\xAleft,\ymin) -- (\xAleft,\Rlow);
\draw[dashed, Acol] (\xAright,\ymin) -- (\xAright,\Rhigh);

\draw[line width=2pt, Acol] (\xAleft,\ymin) -- (\xAright,\ymin);
\node[Acol, below=8pt] at ({0.5*(\xAleft+\xAright) + 1.6},\ymin + 0.9) {$A_J$};

\draw[dashed, Ecol] (\xEleft,\ymin) -- (\xEleft,2.62);
\draw[dashed, Ecol] (\xEright,\ymin) -- (\xEright,2.88);

\draw[line width=3pt, Ecol] (\xEleft,0.18) -- (\xEright,0.18);
\node[Ecol, below=18pt] at ({0.5*(\xEleft+\xEright)},\ymin+1.45) {$E_J$};

\draw[->, thick] (5.20,4.25) -- (5.72,2.88);

\node[align=center, fill=white, inner sep=2pt] at (6.85,4.45)
{$|\Phi'(s)|\ge c\ep R^{-1/2}$\\
on $E_J$};

\end{tikzpicture}

\caption{
The length estimate on one monotonicity interval of $\Phi$. The condition
$|\Phi(s)-R|\le Ch$ confines $\Phi(s)$ to an interval of size $O(h)$, while
on the transverse set $E_J \subset A_J$ we have $|\Phi'(s)|\ge c\ep R^{-1/2}$.
Thus $|E_J|\lesssim h/(\ep R^{-1/2})=\ep^{-1}R^{1/2-\alpha}$.
}
\label{fig:phi-band}
\end{figure}
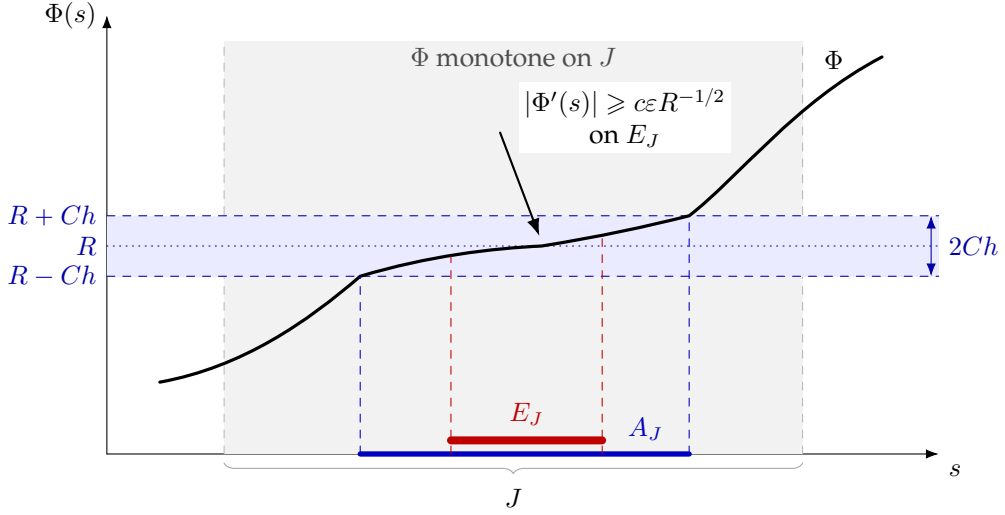

Fix one such monotonicity interval $J$ of $\Phi$ (see Figure \ref{fig:phi-band}). We treat the case
where $\Phi$ is increasing on $J$; the decreasing case is identical. 
Define
$$A_J := \left\{s\in J: |\Phi(s)-R|\le Ch
\right\},$$
and the transverse subset 
$$ E_J
:= \left\{ s\in J: |\Phi(s)-R|\le Ch \ \text{and}\ |\Phi'(s)|\ge c\ep R^{-1/2} \right\}.$$
Since $\Phi$ is monotone on $J$, the set $A_J$ is an interval. Write $A_J=[a,b]$. Then, $$\int_{E_J}|\Phi'(s)|\,ds\le  \int_a^b |\Phi'(s)|\,ds = \int_a^b \Phi'(s)\,ds = \Phi(b)-\Phi(a) \le 2Ch.$$
On the other hand,
$$ \int_{E_J}|\Phi'(s)|\,ds \ge c\ep R^{-1/2}\,|E_J|,$$
by definition of $E_J$. Combining the two estimates, we have $|E_J| \lesssim \ep^{-1}R^{1/2-\alpha}.$ Therefore, on each monotonicity interval of $\Phi$, the set of $s$ such that $|\Phi(s)-R|\lesssim h$ and $|\Phi'(s)|\gtrsim \ep R^{-1/2}$, has length at most $\lesssim \ep^{-1}R^{1/2-\alpha}.$ Since there are only $O(1)$ monotonicity intervals, we obtain
$$
\mathcal H^1(\Gamma_\sigma)
\lesssim
\ep^{-1}R^{1/2-\alpha}.
$$

\medskip
\noindent \textbf{Step 3: From projected arcs to the lattice point count.}

We now pass from the projected length to a lattice point count. Note that a length estimate alone is not sufficient, because a set of small $\h^1$-measure could have many separated components. We will show
that this cannot happen, by showing that the projected set is contained in a union of $O(1)$ arcs on the circle $RS^1$.

Recall that the radial projection argument shows that $\pi(E_\sigma)$ is contained in
$y(G_\sigma)$, where
$$
G_\sigma
:= \left\{ s\in[0,2\pi R): |\Phi(s)-R|\le Ch \ \text{and}\
|\Phi'(s)|\ge c\ep R^{-1/2}
\right\}.
$$
We claim that $G_\sigma$ is a union of $O(1)$ intervals. As shown above, $\Phi$ has only $O(1)$ monotonicity intervals on $[0,2\pi R)$. Moreover, the explicit formula for $\Phi'$ shows that $|\Phi'|$ also has only $O(1)$ monotonicity intervals. Indeed, after rotating so that $\sigma=(d,0)$, one has
$$
|\Phi'(s)|^2
=
\frac{d^2\sin^2(s/R)}
{R^2+d^2+2Rd\cos(s/R)}.
$$
On each half-circle, $x=\cos(s/R)$ is monotone in $s$, and the right-hand
side is a rational function of $x$ with only $O(1)$ critical points.
Therefore $|\Phi'|$ has only $O(1)$ monotonicity intervals.

We can therefore subdivide $[0,2\pi R)$ into $O(1)$ intervals $J$ such
that both $\Phi$ and $|\Phi'|$ are monotone on each $J$. Fix one such
interval $J$. Since $\Phi$ is monotone on $J$, the set
$$
A_J:=\{s\in J:|\Phi(s)-R|\le Ch\}
$$
is an interval. Since $|\Phi'|$ is monotone on $J$, the set
$$
D_J:=\{s\in J:|\Phi'(s)|\ge c\ep R^{-1/2}\}
$$
is also an interval. Hence
$$
G_\sigma\cap J=A_J\cap D_J
$$
is an interval. Since there are only $O(1)$ such intervals
$J$, it follows that $G_\sigma$ is a union of $O(1)$ intervals.

We next estimate the length of $G_\sigma$. Let $J$ be one of the above
intervals. By definition of $A_J$, we have
$\Phi(A_J)\subset [R-Ch,R+Ch].$ Since $\Phi$ is monotone on $J$, it is also monotone on $A_J$, which is an interval. Hence,
$$ \int_{A_J}|\Phi'(s)|\,ds \le 2Ch.$$

Since $G_\sigma\cap J\subset A_J$ and
$|\Phi'(s)|\ge c\ep R^{-1/2}$ on $G_\sigma\cap J$, we get
$$ c\ep R^{-1/2}|G_\sigma\cap J| \le \int_{G_\sigma\cap J}|\Phi'(s)|\,ds \le \int_{A_J}|\Phi'(s)|\,ds \le 2Ch.$$
Thus,
$$
|G_\sigma\cap J| \lesssim \ep^{-1}R^{1/2}h
= \ep^{-1}R^{1/2-\alpha}.$$
Summing over the $O(1)$ intervals $J$, we obtain
$$
|G_\sigma|
\lesssim
\ep^{-1}R^{1/2-\alpha}.
$$
Since $y$ is parametrized by arclength, this gives
$$
\h^1(y(G_\sigma))
\le
|G_\sigma|
\lesssim
\ep^{-1}R^{1/2-\alpha}.
$$

Now $G_\sigma$ is a union of $O(1)$ intervals, so $y(G_\sigma)$ is a union
of $O(1)$ arcs on $RS^1$. Write
$$
y(G_\sigma)=\Gamma_1\cup\cdots\cup\Gamma_M
$$
as a union of connected arc components. Since $G_\sigma$ is a union of
$O(1)$ intervals, we have $M=O(1)$. The arcs $\Gamma_m$ are pairwise
disjoint up to endpoints, and endpoints have $\h^1$-measure zero.
Therefore,
$$ \sum_{m=1}^M \h^1(\Gamma_m)
= \h^1(y(G_\sigma)).$$
We first prove the estimate for a single arc $\Gamma$. Let $\ell:=\h^1(\Gamma).$ Partition $\Gamma$ into $\lceil \ell\rceil$ subarcs, each of length at most $1$. Each such subarc is contained in a Euclidean ball of
radius $1$. Therefore its $C_0$-neighborhood is contained in a Euclidean
ball of radius $C_0+1$, and hence has area at most $\pi(C_0+1)^2.$
Summing over the subarcs gives
$$
|[\Gamma]_{C_0}|
\lesssim \lceil \ell\rceil (C_0+1)^2
\lesssim_{C_0}
1+\ell = 1+\h^1(\Gamma).
$$

Applying this to each $\Gamma_m$, we get
$$
|[y(G_\sigma)]_{C_0}|
\le
\sum_{m=1}^M |[\Gamma_m]_{C_0}|
\lesssim_{C_0}
\sum_{m=1}^M \left(1+\h^1(\Gamma_m)\right).
$$
Since $M=O(1)$, this implies
$$
|[y(G_\sigma)]_{C_0}| \lesssim_{C_0} 1+\sum_{m=1}^M \h^1(\Gamma_m) \lesssim 1+\h^1(y(G_\sigma)).
$$

Finally, since every $x\in E_\sigma$ satisfies $|x-\pi(x)|\le h$ and
$\pi(x)\in y(G_\sigma)$, we have $E_\sigma\subset [y(G_\sigma)]_h.$ 
Let $Q:=\left[-\frac12,\frac12\right]^2.$ Since $h\le 1$, there is an absolute constant $C_0$ such that $$E_\sigma+Q
\subset [y(G_\sigma)]_{h+\sqrt2/2}
\subset [y(G_\sigma)]_{C_0}.$$ Hence,
$$\#(E_\sigma\cap\mathbb Z^2) = \sum_{\lambda\in E_\sigma\cap\mathbb Z^2}|\lambda+Q| \le
|E_\sigma+Q| \le |[y(G_\sigma)]_{C_0}|.$$
Combining this with the previous estimates gives
$$
\#(E_\sigma\cap\mathbb Z^2) \lesssim
1+\h^1(y(G_\sigma)) \lesssim 1+\ep^{-1}R^{1/2-\alpha}.$$
Since $0\le \alpha\le 1/2$ and $R\ge 1$, we have
$R^{1/2-\alpha}\ge 1$, and therefore
$$\#(E_\sigma\cap\mathbb Z^2) \lesssim
\ep^{-1}R^{1/2-\alpha}.$$

\end{proof}

\section{A Rough Damping Function in $H^{1/2}(\T^2)$ satisfying GCC}
\label{section:rough-damping-example}

The Logvinenko--Sereda theorem of Green, Mitkovski, and the first author gives observability from a fixed neighborhood of a set satisfying GCC. Consequently, if a damping function is bounded below by a positive constant on such a neighborhood, then their result directly yields an
estimate of the form \eqref{eq:annular-main-introduction}.

We now show that Theorem \ref{thm:annular-observability} applies in a genuinely different situation. We construct a non-negative damping function
$\gamma\in H^{1/2}(\T^2)\cap L^\infty(\T^2)$ which satisfies GCC, but whose zero set contains a dense open set. At the same time, $\gamma=1$ on a set of positive measure. Thus, $\gamma$ is not bounded below by a positive constant on any nonempty open subset of $\T^2$. The example therefore cannot be reduced to observability from an open region on which the damping is uniformly positive.

\begin{proposition}
\label{prop:rough-damping-example}
There exist a non-negative function $\gamma\in H^{1/2}(\T^2)\cap L^\infty(\T^2)$ and a constant $c_0>0$ such that
\begin{equation}
\label{eq:rough-damping-integral-gcc}
\inf_{x\in\T^2}\inf_{\nu\in S^1}
\int_0^1 \gamma(x+t\nu)\,dt
\ge c_0.
\end{equation}
Moreover, there exist a dense open set $Z\subset\T^2$ and a measurable set
$A\subset\T^2$ of positive measure such that $\gamma=0$ on $Z$ and $\gamma=1$ on $A$. Consequently, $\gamma$ is not a.e. constant, its $L^\infty$-equivalence class has no continuous representative, and $\gamma$ is not bounded below a.e. by a positive constant on
any nonempty open subset of $\T^2$.
\end{proposition}

We note that the integral GCC (\ref{eq:rough-damping-integral-gcc}) implies GGCC. Indeed, the same integral lower bound holds uniformly for the mollifications $\gamma_\ep$, so this follows immediately from Lemma \ref{lemma:ggcc-mollification}. In particular, \eqref{eq:rough-damping-integral-gcc} implies the GGCC required in Theorem \ref{thm:annular-observability}.

The rest of this section is devoted to the construction of the $\gamma$ in Proposition \ref{prop:rough-damping-example}. Define dyadic squares $$Q_{n,k} := \left[\frac{k_1}{2^n}, \frac{k_1+1}{2^n} \right) \times \left[\frac{k_2}{2^n}, \frac{k_2+1}{2^n} \right)$$
of side length $2^{-n}$ where $k = (k_1,k_2) \in \{0, 1, \ldots, 2^{n}-1\}^2$. Let $c_{n,k}$ denote the center of $Q_{n,k}$. Consider $\psi \in C_c^\infty(\R^2)$ satisfying $0 \le \psi\le 1$, $\psi \equiv 1$ on $B(0, \frac{1}{2})$, and $\psi\equiv 0$ on $\R^2\setminus B(0,1)$. Let $r_n = a2^{-3n}$ for small enough $a > 0$, and define $\phi_n: \T^2 \to [0,1]$ by
$$\phi_n(x) := \sum_{k} b_{n,k}(x) =  \sum_{k} \psi\left(\frac{x-c_{n,k}}{r_n} \right),$$
where the sum runs over all the dyadic squares of generation $n$. Note the $b_{n,k}$'s have disjoint supports. Moreover, $\phi_n \equiv 1$ on $\bigcup_k B(c_{n,k}, \frac{r_n}{2})$ and $\operatorname{supp} \phi_n \subset \bigcup_k B(c_{n,k},r_n)$. 

\begin{figure}[ht]
\centering
\begin{tikzpicture}[scale=5]

  \def\n{3}

  \pgfmathtruncatemacro{\N}{2^\n}
  \pgfmathtruncatemacro{\NmOne}{\N-1}
  \pgfmathsetmacro{\h}{1/\N}              
  \pgfmathsetmacro{\rad}{pow(2,-3*\n)}    

  \draw[thick] (0,0) rectangle (1,1);
  \draw[step=\h, gray!60, very thin] (0,0) grid (1,1);

  \foreach \i in {0,...,\NmOne} {
    \foreach \j in {0,...,\NmOne} {
      \pgfmathsetmacro{\cx}{(\i+0.5)/\N}
      \pgfmathsetmacro{\cy}{(\j+0.5)/\N}
      \fill[blue!50, opacity=0.7] (\cx,\cy) circle[radius=\rad];
      \fill[blue!90] (\cx,\cy) circle[radius=0.4pt];
    }
  }

  \node[below left] at (0,0) {$0$};
  \node[below] at (1,0) {$1$};
  \node[left] at (0,1) {$1$};

\end{tikzpicture}
\caption{The support of $\phi_n$ is contained in the union of the balls $B(c_{n,k},r_n)$, one centered in each dyadic square $Q_{n,k}$.}
\end{figure}
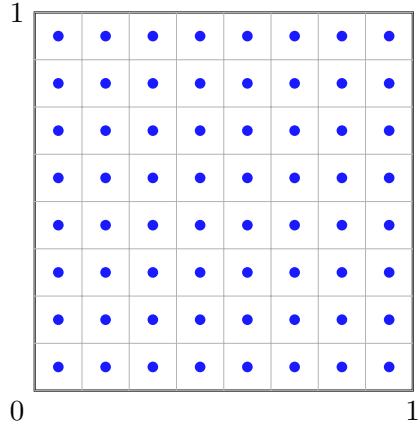

Now, define $$\gamma_N(x) := \prod_{n=1}^N (1 - \phi_n(x)),$$
and consider the pointwise limit
$$\gamma(x) := \lim_{N\to\infty} \gamma_N(x).$$
As $0 \le 1 - \phi_n(x) \le 1$ for each $n$, the $\gamma_N$'s are decreasing and bounded below, and so the pointwise limit exists.

Next, we derive some line intersection estimates that lead us to the geometric control condition.

\begin{lemma}
\label{line-estimate-for-phi_n}
There exists an absolute constant $C>0$ such that, for every
$n\ge1$, every $x\in\T^2$, and every $\nu\in S^1$,
$$
\int_0^1 \phi_n(x+t\nu)\,dt
\le C2^nr_n.
$$
In particular,
$$
\int_0^1 \phi_n(x+t\nu)\,dt
\le Ca2^{-2n},
$$
when $r_n=a2^{-3n}$.
\end{lemma}
\begin{proof}
    Consider a line segment $\ell$ of unit length in $\T^2$ that makes an angle $\theta \in [0, \pi/2]$ with the vertical. Let $N_v$ and $N_h$ denote the number of vertical and horizontal grid lines (of length $2^{-n}$) that meet $\ell$, respectively. Then, $(N_v - 1)\, 2^{-n} \le \sin \theta \le 1$, giving $N_v \le 2^n + 1$. Similarly, $N_h \le 2^n + 1$. Thus, $\ell$ intersects at most $N_v + N_h + 1 \lesssim 2^n$ dyadic squares.
    
    Next, recall that $\phi_n$ is supported on $\bigcup_k B(c_{n,k}, r_n)$, and consider the trivial bound $\h^1(\ell \cap B(c_{n,k}, r_n)) \le 2r_n$. As $\phi_n \le 1$, we have $\int_\ell \phi_n d\h^1 \lesssim 2r_n \cdot 2^n$ which is the desired estimate.
\end{proof}

We will use this to obtain a uniform lower bound on the line integral $\int_\ell \gamma\, d\h^1$. 

\begin{lemma}
\label{line-estimate-for-gamma}
There exists $c_0>0$ such that
\begin{equation}
\label{eq:line-estimate-for-gamma}
\inf_{x\in\T^2}\inf_{\nu\in S^1} \int_0^1 \gamma(x+t\nu)\,dt \ge c_0.
\end{equation}
\end{lemma}

\begin{proof}
Fix $x\in\T^2$ and $\nu\in S^1$. For $0\le y_1,\ldots,y_N\le1$, one has $$ \prod_{n=1}^N(1-y_n)\geq1-\sum_{n=1}^Ny_n.$$ Applying this pointwise with $y_n=\phi_n(x+t\nu)$ and integrating over $t\in[0,1]$, we obtain $$ \int_0^1\gamma_N(x+t\nu)\,dt
\ge 1-\sum_{n=1}^N\int_0^1\phi_n(x+t\nu)\,dt. $$
By Lemma \ref{line-estimate-for-phi_n},
$$ \int_0^1\gamma_N(x+t\nu)\,dt \ge 1 - Ca\sum_{n=1}^\infty2^{-2n}.$$
Choose $a>0$ sufficiently small that
$$c_0 := 1-Ca\sum_{n=1}^\infty2^{-2n} >0.$$
Since $0\le\gamma_N\le1$ and $\gamma_N\to\gamma$ pointwise, dominated convergence gives
$$ \int_0^1\gamma(x+t\nu)\,dt = \lim_{N\to\infty} \int_0^1\gamma_N(x+t\nu)\,dt
\ge c_0. $$
\end{proof}

Lastly, we show that $\gamma\in H^{1/2}(\T^2)$. A direct computation gives $\|b_{n,k}\|_{L^2}=r_n\|\psi\|_{L^2}$, and 
$\|\nabla b_{n,k}\|_{L^2}=\|\nabla\psi\|_{L^2}$. Since the functions $b_{n,k}$ have disjoint supports for each fixed $n$, $\|\phi_n\|_{L^2}^2\lesssim 2^{2n}r_n^2$, and $\|\nabla\phi_n\|_{L^2}^2\lesssim 2^{2n}.$
Hence, by interpolation,
$$
\|\phi_n\|_{\dot H^{1/2}}^2
\lesssim
\|\phi_n\|_{L^2}\|\nabla\phi_n\|_{L^2}
\lesssim 2^{2n}r_n
=
a2^{-n},
$$
and therefore
$$
\sum_{n=1}^\infty \|\phi_n\|_{\dot H^{1/2}}<\infty.
$$
On $\T^2$, the Fourier and Gagliardo seminorms are equivalent: $[f]_{H^{1/2}(\T^2)}
\approx \|f\|_{\dot H^{1/2}(\T^2)}.$ Consequently,
$$
\sum_{n=1}^\infty [\phi_n]_{H^{1/2}(\T^2)}
<\infty.
$$

For $a_n,b_n\in[0,1]$,
$$
\left|\prod_{n=1}^N a_n-\prod_{n=1}^N b_n\right|
\le
\sum_{n=1}^N|a_n-b_n|.
$$
Applying this with $a_n=1-\phi_n(x)$ and $b_n=1-\phi_n(y)$, and using Minkowski's inequality for the Gagliardo seminorm, gives
$$ [\gamma_N]_{H^{1/2}(\T^2)} \le \sum_{n=1}^N[\phi_n]_{H^{1/2}(\T^2)}.$$
Thus, by Fatou's lemma,
$$ [\gamma]_{H^{1/2}(\T^2)}
\le \liminf_{N\to\infty}\,[\gamma_N]_{H^{1/2}(\T^2)}
\le \sum_{n=1}^\infty[\phi_n]_{H^{1/2}(\T^2)} <\infty.
$$
Since $0\le\gamma\le1$, we also have $\gamma\in L^2(\T^2)$, and hence $\gamma\in H^{1/2}(\T^2)$.

Finally, we discuss the roughness properties of $\gamma$. Let $$Z:=\bigcup_{n=1}^\infty\bigcup_k B(c_{n,k},r_n/2).$$
Since $\phi_n\equiv1$ on $B(c_{n,k},r_n/2)$, we have $\gamma=0$ on $Z$. The set $Z$ is open and dense: every nonempty open set contains a dyadic square $Q_{n,k}$ for all sufficiently large $n$, and hence contains some ball $B(c_{n,k},r_n/2)$. On the other hand, since
$$|\{\phi_n>0\}|\lesssim 2^{2n}r_n^2=a^2\, 2^{-4n},$$
we may choose $a>0$ sufficiently small that
$$\sum_{n=1}^\infty |\{\phi_n>0\}|<1.$$
Thus,
$$A:=\T^2 \,\setminus\bigcup_{n=1}^\infty \{\phi_n>0\}$$
has positive measure. For every $x\in A$, we have $\phi_n(x)=0$ for all $n\ge 1$, and therefore $\gamma(x)=1$. Hence, $|\{x\in \T^2: \gamma(x)=1\}|>0.$

Since $\gamma=0$ on $Z$ and $\gamma=1$ on the positive-measure set $A$, $\gamma$ is not a.e. constant. Moreover, every nonempty open set $U\subset\T^2$ meets the dense open set
$Z$ in a nonempty open set. Thus $U$ contains a set of positive measure
on which $\gamma=0$, and consequently $\gamma$ cannot be bounded below
almost everywhere by a positive constant on $U$.

Finally, the $L^\infty$ equivalence class of $\gamma$ has no continuous representative. Indeed, suppose that
$\widetilde\gamma\in C(\T^2)$ and
$\widetilde\gamma=\gamma$ almost everywhere. On every open ball
$B\subset Z$, one has $\widetilde\gamma=0$ almost everywhere. By
continuity, $\widetilde\gamma=0$ everywhere on $B$. Hence
$\widetilde\gamma=0$ on the dense set $Z$, and continuity gives $\widetilde\gamma\equiv0$. This contradicts $\gamma=1$ on the positive-measure set $A$. This completes the example.

\section{Trigonometric Polynomials with Spectra Near a Convex Polygon}

In this section, we establish an observability inequality for functions in $L^2(\T^2)$ with Fourier support contained in the neighborhood of a convex polygon $P\subset \R^2$. We denote the vertices of $P$ by $v_1, \ldots, v_{N}$ and the sides by $S_1,\ldots,S_N$ ordered cyclically. Let $\nu_j$ be the unit vector perpendicular to $S_j$ and define $\Theta := \{\nu_1, \ldots, \nu_N\}$. For $R \ge 1$, let $\Lambda_R(P)$ denote the lattice points in a fixed-width neighborhood of $R\partial P$ (the $R$-dilate of the boundary of $P$) i.e.,
$$\Lambda_R(P) := \{\xi \in \Z^2: \operatorname{dist}(\xi, R\partial P) \le 1\}.$$
With this notation, we state our main result below:

\begin{tcolorbox}[colback=SeaGreen!10, colframe=SeaGreen!80, boxrule=0.5mm, left=2mm, right=2mm, boxsep=1mm, arc=1mm.]
\begin{theorem}
\label{main-result-polygon}
Let $P\subset\mathbb R^2$ be a convex polygon, and let
$\gamma\in L^\infty(\mathbb T^2)$ be non-negative. Suppose that $\gamma$ satisfies the generalized geometric control condition (GGCC) in the normal direction of every side of $P$: there exist $L,c_0>0$ such that
$$
\liminf_{\delta\to0}
\inf_{x\in\mathbb T^2} \inf_{\nu \in \Theta}
\frac{1}{|\Gamma_{x,\nu,L,\delta}|}
\int_{\Gamma_{x,\nu,L,\delta}}\gamma(y)\,dy
\ge c_0.
$$

Then there exists a constant $C>0$, depending on $P$ and $\gamma$ but independent of $R$, such that 
$$\|f\|_{L^2(\mathbb T^2)}^2 \le C\int_{\mathbb T^2}|f(x)|^2\gamma(x)\,dx$$
for every $R\geq 1$ and every $f\in L^2(\mathbb T^2)$ satisfying $$\operatorname{supp}\widehat f\subset\Lambda_R(P).$$
\end{theorem}
\end{tcolorbox}

\textit{Proof Sketch.}
Before proving the theorem, we explain the mechanism. The proof is by contradiction. If the desired estimate failed, we could find a sequence of normalized functions $f_n$ with $\operatorname{supp}\widehat{f_n}\subset \Lambda_{R_n}(P)$ and $\|f_n\sqrt{\gamma}\|_2\to 0$. If the radii $R_n$ remain bounded, then the functions lie in a fixed finite-dimensional space, and compactness gives a nonzero trigonometric polynomial vanishing on $\{\gamma>0\}$, which is impossible. Thus, $R_n\to\infty$.

Next, we divide the frequency support into two parts: the portions lying away from the vertices, and the portions lying near the vertices; see Figure \ref{fig:polygon-decomposition}. The estimate away from the vertices is provided by Proposition \ref{polygon-without-corners}, whose proof uses the strip estimate of Theorem \ref{two-strip-theorem} together with an almost-orthogonality argument between different sides. Thus, any counterexample must place a definite amount of its Fourier mass near one of the vertices. We then translate that vertex to the origin by a modulation. After passing to a limit, the local geometry near the vertex becomes the cone generated by the two sides meeting there. The limiting function is therefore supported in a translate of this cone, while still vanishing on the positive measure set $\{\gamma>0\}$. The cone uniqueness result in Proposition \ref{convex-cone} rules this out. This contradiction completes the argument.

    \begin{figure}[ht]
\centering
\begin{tikzpicture}[scale=0.65, line cap=round, line join=round]

\definecolor{SideSky}{RGB}{186,222,245}
\definecolor{SideSkyLabel}{RGB}{35,105,175}
\definecolor{CornerLightRed}{RGB}{205,100,80}
\definecolor{CornerLightRedLabel}{RGB}{130,45,35}
\definecolor{GridSoft}{RGB}{220,220,220}

\colorlet{gridgray}{GridSoft}
\colorlet{sideblue}{SideSky}
\colorlet{sidelabel}{SideSkyLabel}
\colorlet{cornerorange}{CornerLightRed}
\colorlet{cornerlabel}{CornerLightRedLabel}

\def\Rout{6}     
\def\Rin{4.5}      
\def\Rsmall{2.10}   
\def\tcut{0.23}     
\pgfmathsetmacro{\omt}{1-\tcut}

\foreach \x in {-8,-7,...,8} {
  \draw[gridgray, dashed, line width=0.35pt] (\x,-8) -- (\x,8);
}
\foreach \y in {-8,-7,...,8} {
  \draw[gridgray, dashed, line width=0.35pt] (-8,\y) -- (8,\y);
}

\foreach \x in {-8,-7,...,8} {
  \foreach \y in {-8,-7,...,8} {
    \fill[gray!55] (\x,\y) circle (1.15pt);
  }
}

\foreach \i/\ang in {0/90,1/18,2/-54,3/-126,4/162} {
  \coordinate (O\i) at (\ang:\Rout);
  \coordinate (I\i) at (\ang:\Rin);
  \coordinate (P\i) at (\ang:\Rsmall);
}

\foreach \i/\j in {0/1,1/2,2/3,3/4,4/0} {
  \coordinate (Os\i) at ($(O\i)!\tcut!(O\j)$);
  \coordinate (Oe\i) at ($(O\i)!\omt!(O\j)$);
  \coordinate (Is\i) at ($(I\i)!\tcut!(I\j)$);
  \coordinate (Ie\i) at ($(I\i)!\omt!(I\j)$);
}

\foreach \i in {0,1,2,3,4} {
  \fill[sideblue] (Os\i)--(Oe\i)--(Ie\i)--(Is\i)--cycle;
}

\fill[cornerorange] (Oe4)--(O0)--(Os0)--(Is0)--(I0)--(Ie4)--cycle;
\fill[cornerorange] (Oe0)--(O1)--(Os1)--(Is1)--(I1)--(Ie0)--cycle;
\fill[cornerorange] (Oe1)--(O2)--(Os2)--(Is2)--(I2)--(Ie1)--cycle;
\fill[cornerorange] (Oe2)--(O3)--(Os3)--(Is3)--(I3)--(Ie2)--cycle;
\fill[cornerorange] (Oe3)--(O4)--(Os4)--(Is4)--(I4)--(Ie3)--cycle;

\begin{scope}
  \clip (O0)--(O1)--(O2)--(O3)--(O4)--cycle;
  \foreach \x in {-8,-7,...,8} {
    \foreach \y in {-8,-7,...,8} {
      \fill[black] (\x,\y) circle (0.95pt);
    }
  }
\end{scope}

\begin{scope}
  \clip (I0)--(I1)--(I2)--(I3)--(I4)--cycle;
  \foreach \x in {-8,-7,...,8} {
    \foreach \y in {-8,-7,...,8} {
      \fill[gray!55] (\x,\y) circle (1.15pt);
    }
  }
\end{scope}

\foreach \i in {0,1,2,3,4} {
  \draw[dashed, line width=0.8pt] (Os\i)--(Is\i);
  \draw[dashed, line width=0.8pt] (Oe\i)--(Ie\i);
}

\draw[thick] (O0)--(O1)--(O2)--(O3)--(O4)--cycle;
\draw[thick] (I0)--(I1)--(I2)--(I3)--(I4)--cycle;
\draw[thick] (P0)--(P1)--(P2)--(P3)--(P4)--cycle;

\node at (0,-6.6) {$\Lambda_R(P) := \{\xi\in \Z^2: \operatorname{dist}(\xi, R\partial P) \le 1\}$};
\node at (1.95,1.35) {$\partial P$};

\node[text=sidelabel] at (-5,5.4)
  {\large side strips $\Lambda_R^{(A)}(P)$};

\draw[sidelabel, line width=1pt, ->]
  (-4.5,4.95) -- (-3.3,3.5);

\node[text=cornerlabel] at (3.6,6.5)
  {\large corner pieces $\mathcal C_{R,k}^{(A)}$};

\draw[cornerlabel, line width=1pt, ->]
  (1.4,6.2) -- (0.5,5.5);

\end{tikzpicture}
\caption{Decomposition of $\Lambda_R(P)$ into side strips $\Lambda_R^{(A)}(P)$ and corner pieces $\bigcup_{k=1}^N \mathcal C_{R,k}^{(A)}$.}
\label{fig:polygon-decomposition}
\end{figure}
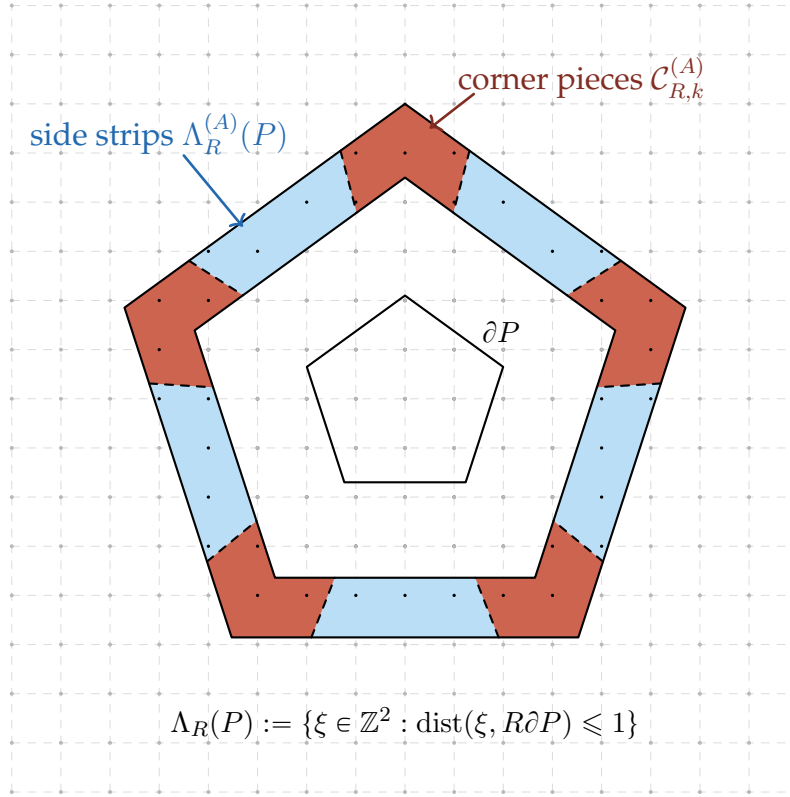

\medskip
Now, we formally define the portion of the frequency support lying away from the vertices. Note that  $S_j=[v_j,v_{j+1}]$ with indices taken modulo $N$. Let $\ell_j$ denote the line containing the segment $S_j$, and for $R\ge 1$, let $\pi_{j,R}$ denote the orthogonal projection onto the line $R\ell_j$. For $A\ge 1$, define
$$ \widetilde\Gamma_{j,R}^{(A)}
:= \left\{\xi\in \mathbb Z^2:
\operatorname{dist}(\xi,R\ell_j)\le 1,\, \pi_{j,R}\xi\in RS_j,\, \operatorname{dist}(\pi_{j,R}\xi, Rv_j)\ge A, \operatorname{dist}(\pi_{j,R}\xi, Rv_{j+1})\ge A
\right\}.$$

We define
$$ \Lambda_R^{(A)}(P)
:= \bigcup_{j=1}^N \widetilde\Gamma_{j,R}^{(A)}.$$
The sets $\widetilde\Gamma_{j,R}^{(A)}$ may not be disjoint, so we choose a disjoint decomposition:
$$ \Lambda_R^{(A)}(P)
= \bigcup_{j=1}^N \Gamma_{j,R}^{(A)},$$
where $\Gamma_{j,R}^{(A)}\subset \widetilde\Gamma_{j,R}^{(A)}$.

\begin{proposition}
\label{polygon-without-corners}
    Assume that $\gamma$ satisfies the GGCC in the normal direction to each side of the convex polygon $P$. Then there exists $A_0\ge 1$ such that for every $A\ge A_0$ there is a constant
$$C_A(P,\gamma)<\infty$$
with the following property.

Let $\Lambda_R^{(A)}(P)\subset \Lambda_R(P)$ be obtained by removing tangential neighborhoods of size $A$ around every vertex of $R\partial P$ along the adjacent sides.
Then every $g \in L^2(\T^2)$ with
$\operatorname{supp} \widehat{g} \subset \Lambda_R^{(A)}(P)$ satisfies
$$
\|g\|_{L^2(\T^2)}^2 \le C_A(P,\gamma)\int_{\T^2} |g|^2\gamma,
$$
uniformly in $R$.
\end{proposition}

\begin{proposition}
\label{convex-cone}
    Let $C_0\subset \R^2$ be a closed convex cone with vertex at the origin, nonempty interior, and opening angle strictly smaller than $\pi$. Let $C = a + C_0$ for some $a\in \R^2$.
Suppose
$$
f(x)=\sum_{\xi\in C\cap \Z^2}\widehat f(\xi)e^{2\pi i \xi\cdot x}\in L^2(\T^2),
$$
vanishes on a set of positive measure. Then, $f \equiv 0$. 
\end{proposition}

Note that if $\gamma \in L^\infty(\T^2)$ satisfies the GGCC, then $\{\gamma>0\}$ has positive measure, so the hypothesis of Proposition \ref{convex-cone} is satisfied for functions vanishing on $\{\gamma > 0\}$. First, we show how Theorem \ref{main-result-polygon} follows from the above results.

\begin{proof}[Proof of Theorem \ref{main-result-polygon}]
    Suppose, for the sake of contradiction, that Theorem \ref{main-result-polygon} fails. Then, there exists a sequence $R_n \ge 1$, and functions $f_n \in L^2(\T^2)$ with $\operatorname{supp} \widehat{f_n} \subset \Lambda_{R_n}(P)$ and $\|f_n\|_{L^2(\T^2)} = 1$ such that $$\|f_n\sqrt{\gamma}\|^2_{L^2(\T^2)} =  \int_{\T^2} |f_n(x)|^2\, \gamma(x) \xrightarrow{n\to\infty} 0.$$

\medskip
\noindent\textbf{Step 1: Reduction to high frequencies.}
    
    If $\sup_n R_n < \infty$, then all frequencies lie in a bounded subset of $\Z^2$, so the functions $f_n$ lie in a finite dimensional subspace of $L^2(\T^2)$. After passing to a subsequence, we have $f_n \to f$ in $L^2$ with $\|f\|_{L^2(\T^2)} = 1$ by the Bolzano-Weierstrass theorem. As $\|f_n \sqrt{\gamma}\|_{L^2(\T^2)} \to 0$ as $n\to\infty$ and $\gamma \in L^\infty(\T^2)$, we also get $\|f \sqrt{\gamma}\|_{L^2(\T^2)} = 0$. Thus, $f = 0$ a.e. on $\{\gamma > 0\}$. However, a trigonometric polynomial on $\T^2$ cannot vanish on a set of positive measure, so this is a contradiction. Therefore, $R_n \to \infty$ as $n\to\infty$. 

\medskip
\noindent\textbf{Step 2: A counterexample must concentrate near a vertex.}

    Fix $A \ge A_0$ as in Proposition \ref{polygon-without-corners}. Decompose
    $$\Lambda_{R_n}(P) = \Lambda_{R_n}^{(A)}(P) \cup \mathcal{C}^{(A)}_{R_n,1} \cup \ldots \cup\mathcal{C}^{(A)}_{R_n,N},$$
    and accordingly decompose each $f_n$ as
    $$f_n = b_n + \sum_{k=1}^N c_{n,k},$$
    where $\operatorname{supp} \widehat{b_n} \subset \Lambda_{R_n}^{(A)}(P)$ and $\operatorname{supp} \widehat{c_{n,k}} \subset \mathcal{C}_{R_n,k}^{(A)}$ for each $1\le k\le N$. We have 
    $$\|f_n\|^2_{L^2(\T^2)} = \|b_n\|^2_{L^2(\T^2)} + \sum_{k=1}^N \|c_{n,k}\|^2_{L^2(\T^2)}.$$

    We claim that some corner carries positive mass as $n\to\infty$, i.e., $$\|c_{n,k_0}\|_{L^2(\T^2)} \gtrsim 1$$ for some $1\le k_0 \le N$, and we denote the corresponding vertex by $v := v_{k_0}$. On the contrary, suppose $\|c_{n,k}\|_{L^2(\T^2)} \to 0$ as $n\to\infty$ for each $1\le k\le N$. As $\|f_n\|_{L^2(\T^2)} = 1$, it follows that $\|b_n\|_{L^2(\T^2)} \to 1$ as $n\to\infty$. Proposition \ref{polygon-without-corners} applied to $b_n$ gives $$\|b_n\|_{L^2(\T^2)} \lesssim \|b_n \sqrt{\gamma}\|_{L^2(\T^2)},$$ so that $\|b_n \sqrt{\gamma}\|_{L^2(\T^2)} \gtrsim 1$ for all large $n$. However, 
    $$\|b_n \sqrt{\gamma}\|_{L^2(\T^2)} \le \|f_n \sqrt{\gamma}\|_{L^2(\T^2)} + \sum_{k=1}^N \|c_{n,k} \sqrt{\gamma}\|_{L^2(\T^2)} \xrightarrow{n\to\infty} 0,$$
    a contradiction. The claim follows.

\medskip
\noindent\textbf{Step 3: Translate the relevant vertex to the origin.}

    Now, we shall translate this corner in frequency space, such that the translated corner is contained in a bounded set near the origin $(0,0)$, independent of $n$. Choose lattice points $m_n$ closest to $R_n v_{k_0}$ so that $$|m_n - R_n v_{k_0}| \lesssim 1,$$ and define $$g_n(x) := e^{-2\pi i m_n \cdot x}\, f_n(x).$$ Then, $\|g_n\|_{L^2(\T^2)}= 1$ and $\|g_n \sqrt{\gamma}\|_{L^2(\T^2)} \to 0$. Also define $$d_{n,k_0}(x) := e^{-2\pi i m_n \cdot x}\, c_{n,k_0}(x),$$ so that $\operatorname{supp} \widehat{d_{n,k_0}} \subset \operatorname{supp} \widehat{c_{n,k_0}} - m_n$. Then, for each $n$, $\operatorname{supp} \widehat{d_{n,k_0}}$ lies in a bounded set $K$. Consequently, the $d_{n,k_0}$'s lie in a finite dimensional subspace of $L^2(\T^2)$ and we can pass to a subsequence to get $$d_{n,k_0} \to d_{k_0}$$ in $L^2(\T^2).$ We have $\|d_{k_0}\|_{L^2(\T^2)} \gtrsim 1$. 

\medskip
\noindent\textbf{Step 4: Extract a nonzero weak limit.}

    Since $\|g_n\|_{L^2(\T^2)}=1$, we have $g_n \rightharpoonup g$ weakly after passing to a subsequence for some $g\in L^2(\T^2)$. Define the orthogonal projection $\mathcal{F}_K: L^2(\T^2) \to L^2(\T^2)$ by $$(\mathcal{F}_K f)(x) = \sum_{\xi \in K \cap \Z^2} \widehat{f}(\xi) \,e^{2\pi i\xi\cdot x}.$$ Then, we have $\mathcal F_K g_n \rightharpoonup \mathcal F_K g$ weakly in $L^2(\T^2)$. On the other hand, $\mathcal F_K g_n$ is bounded, and belongs to the  finite-dimensional space $\operatorname{span}\{e^{2\pi i\xi\cdot x}:\xi\in K\cap\mathbb Z^2\}.$ Hence, after
passing to a further subsequence, $\mathcal{F}_K g_n \to h$ strongly. Since strong convergence implies weak convergence, and weak limits are unique, we must have $h=\mathcal F_K g$. Finally,
$$ \|\mathcal F_K g_n\|_{L^2(\T^2)}
\ge \|d_{n,k_0}\|_{L^2(\T^2)} \gtrsim 1,$$
and so 
$$\|\mathcal{F}_K\, g\|_{L^2(\T^2)} = \|h\|_{L^2(\T^2)} = \lim_{n\to\infty} \|\mathcal F_K g_n\|_{L^2(\T^2)} \gtrsim 1.$$
Therefore, $g \not\equiv 0$.

\medskip
\noindent\textbf{Step 5: The limit has spectrum in a translated tangent cone.}

    We show that $\operatorname{supp} \widehat{g} \subset C \cap \Z^2$, where $C$ is a closed convex cone with opening angle smaller than $\pi$. Define the cone $$T_v = \R^+ (v_{k_0 -1} - v_{k_0}) + \R^+ (v_{k_0 +1} - v_{k_0})$$ and set $\Gamma := \partial P\setminus (E^+ \cup E^-)$, where $E^+$ and $E^-$ denote the two sides of $P$ meeting at $v_{k_0}$. As $v_{k_0} \notin \Gamma$, we have $\delta := \operatorname{dist}(v, \Gamma) > 0$.
    
    Now fix $\xi \in \Z^2$, and suppose $\widehat{g_n}(\xi) \neq 0$ for infinitely many $n$. Then, $\xi + m_n \in \Lambda_{R_n}(P)$ for infinitely many $n$. For each such $n$, there exists $y_n \in R_n \partial P$ such that $|\xi + m_n - y_n| \le 1$. We claim that $y_n \in R_n E^- \cup R_n E^+$ for sufficiently large $n$. If not, then we have $|y_n - R_n v| \ge R_n \delta$ along a subsequence. However, the triangle inequality yields
    $$|y_n - R_n v| \le |\xi| + |\xi + m_n - y_n| + |m_n - R_n v| \le M_0 + 1 + |\xi|,$$
    which is not possible since the left side blows up with $n \to \infty$ while the right side stays bounded. Thus, $y_n \in R_n E^- \cup R_n E^+$ for all sufficiently large $n$. As a result, $y_n - R_n v \in T_v$ and $$|\xi - (y_n - R_n v)| \le |m_n - R_n v| + |\xi + m_n - y_n| \le M_0 + 1,$$
    so if $\widehat{g_n}(\xi) \ne 0$ for infinitely many $n$, we have $\xi \in T_v + B(0, 1+ M_0) \cap \Z^2$. We can choose $q\in \Z^2$ such that $$T_v + B(0,1+M_0) \subset q + T_v.$$ Therefore, if $\xi \notin (q + T_v) \cap \Z^2$, then $\widehat{g_n}(\xi) = 0$ for sufficiently large $n$. As $g_n \rightharpoonup g$ weakly, this gives $\widehat{g}(\xi) = 0$ for such $\xi$. Consequently, $$\operatorname{supp} \widehat{g} \subset (q + T_v) \cap \Z^2.$$

\begin{figure}[htpb]
\centering
\begin{tikzpicture}[
    >=Latex, 
    font=\small, 
    scale=0.9, 
    every node/.style={inner sep=2pt}
]

\colorlet{polygongray}{gray!20}
\colorlet{cornerblue}{blue!15}
\colorlet{coneorange}{orange!20}
\colorlet{windowbg}{blue!5}
\colorlet{windowborder}{blue!60!black}

\begin{scope}[shift={(0,0)}]
    \node[anchor=west, font=\bfseries] at (-4.5, 3.8) {(a)};
    
    \coordinate (RnV) at (0,0);
    \coordinate (P1) at (-2, 2.5);
    \coordinate (P2) at (-4, 1.5);
    \coordinate (P3) at (-4.5, -1.0);
    \coordinate (P4) at (-1.5, -2.5);
    
    \draw[thick, fill=polygongray, line join=round] 
        (RnV) -- (P1) -- (P2) -- (P3) -- (P4) -- cycle;
    \node at (-3.8, 2.2) {$R_n \partial P$};
    
    \begin{scope}
        \clip (RnV) -- (P1) -- (P2) -- (P3) -- (P4) -- cycle;
        \fill[cornerblue] (RnV) circle (1.2);
    \end{scope}
    \node at (-1.75, 0.7) {$\mathcal C_{R_n,k_0}^{(A)}$};
    
    \draw[thick] (RnV) -- (P1) node[midway, above right] {$R_nE^+$};
    \draw[thick] (RnV) -- (P4) node[midway, below right] {$R_nE^-$};
    
    \draw[dashed, thick] (RnV) circle (0.7);
    \node[fill=white, inner sep=1pt, rounded corners=2pt] at (1, 1) {$B(R_nv, C)$};
    
    \coordinate (mn) at (0.35, -0.3);
    \fill (mn) circle (1.5pt) node[below right, fill=white, inner sep=1pt] {$m_n$};
    
    \fill (RnV) circle (1.5pt) node[above right, fill=white, inner sep=1pt] {$R_n v$};
    
\end{scope}

\draw[->, thick, shorten >=2pt, shorten <=2pt] (1.2, 0) -- (5.2, 0) 
    node[midway, above] {$\lambda \mapsto \lambda - m_n$} 
    node[midway, below, align=center, font=\footnotesize] {modulation by \\ $e^{-2\pi i m_n \cdot x}$};

\begin{scope}[shift={(9.5,0)}]
    \node[anchor=west, font=\bfseries] at (-5, 3.8) {(b)};
    
    \coordinate (q) at (-0.35, 0.3);
    
    \coordinate (DirP) at (-2, 2.5);
    \coordinate (DirM) at (-1.5, -2.5);
    
    \draw[rounded corners=8pt, dashed, thick, draw=windowborder, fill=windowbg] 
        (-3.5, -2.5) rectangle (1.5, 2.5);
    \node[below right, text=windowborder, font=\bfseries] at (-3.5, 2.45) {$K$};
    
    \begin{scope}
        \clip (-4, -3) rectangle (2.5, 3.5);
        \fill[coneorange, opacity=0.7] 
            (q) -- ($(q) + 3*(DirP)$) -- ($(q) + 3*(DirP) + 3*(DirM)$) -- ($(q) + 3*(DirM)$) -- cycle;
    \end{scope}
    
    \draw[thick] (q) -- ($(q) + 0.9*(DirP)$) 
        node[above left, fill=white, inner sep=1pt, rounded corners=2pt] {$v_{k_0+1}-v$};
    \draw[thick] (q) -- ($(q) + 0.9*(DirM)$) 
        node[below left, fill=white, inner sep=1pt, rounded corners=2pt] {$v_{k_0-1}-v$};
    
    \node[orange!80!black, fill=white, inner sep=1pt] at (-2.4, 0.5) {$q+T_v$};
    
    \draw[->] (-4, 0) -- (2.5, 0) node[right] {$\xi_1$};
    \draw[->] (0, -3.5) -- (0, 3.5) node[above] {$\xi_2$};
    
    \fill (q) circle (1.5pt) node[above right, fill=white, inner sep=1pt] {$q$};
    
    \foreach \x in {-3, -2, -1, 0, 1} {
        \foreach \y in {-2, -1, 0, 1, 2} {
            \fill[gray] (\x, \y) circle (1.2pt);
        }
    }
    
    \fill[black] (-1, 0) circle (2.5pt) 
        node[above right, fill=white, inner sep=1pt, rounded corners=2pt] {$\xi$};
    
    \node[align=center, fill=white, inner sep=2pt, rounded corners=2pt] 
        at (-1.0, -4.1) {$\operatorname{supp}\widehat g \subset (q+T_v)\cap\mathbb Z^2$};
        
\end{scope}

\end{tikzpicture}
\caption{Corner blow-up near the selected vertex.}
\label{fig:corner-blowup}
\end{figure}
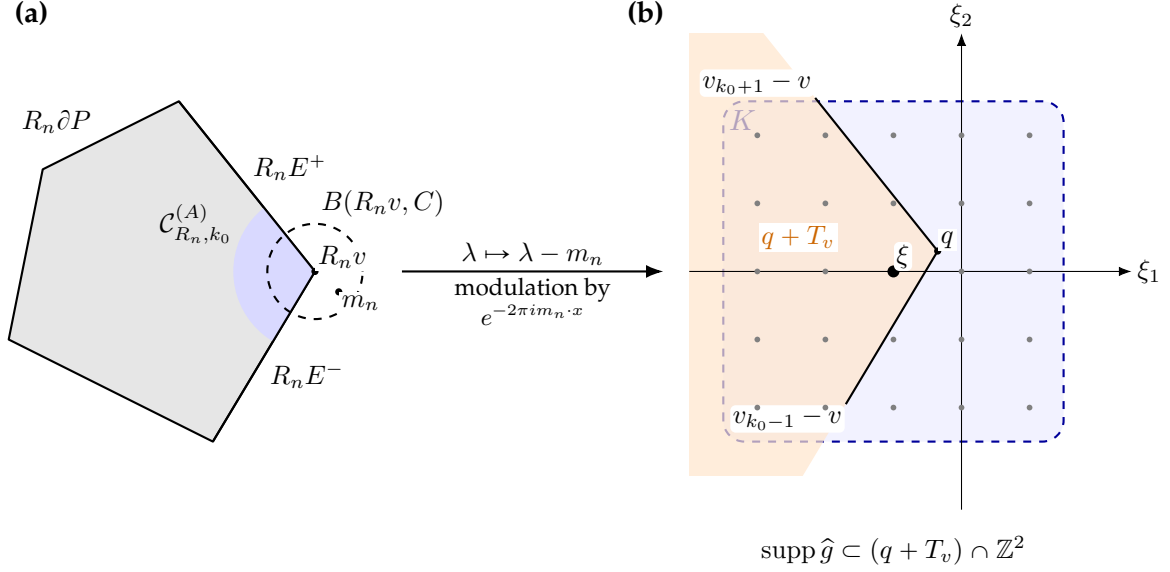

\medskip
\noindent\textbf{Step 6: Cone uniqueness gives the contradiction.}

    Finally, we show that $g = 0$ a.e. on $\{\gamma > 0\}$. As $g_n \rightharpoonup g$ weakly and $\gamma \in L^\infty(\T^2)$, we have $g_n\sqrt{\gamma} \rightharpoonup g \sqrt{\gamma}$ weakly. As $\|g_n \sqrt{\gamma}\|_{L^2} \to 0$ as $n\to \infty$, we must have $g\sqrt{\gamma} = 0$ a.e. on $\T^2$. It follows that $g = 0$ a.e. on $\{\gamma > 0\}$. Using Proposition \ref{convex-cone}, we have $g \equiv 0$. This contradicts $\|\mathcal{F}_K g\|_{L^2(\T^2)} \gtrsim 1$. 
\end{proof}

We shall now turn our attention to the proof of Proposition \ref{polygon-without-corners}, which requires careful bookkeeping. We begin the discussion here, and prove some helpful lemmas before beginning the proof. 

\begin{lemma}
\label{geometric-lemma-1}
Let $P\subset \mathbb R^2$ be a convex polygon, with $\Gamma_{j,R}^{(A)}$ defined as above. Suppose $S_i$ and $S_j$ are not
parallel. Then there exists $c_P>0$, depending
only on $P$, such that for every $R\ge 1$, every $A\ge 1$, and every $(\xi, \zeta) \in \Gamma_{i,R}^{(A)} \times \Gamma_{j,R}^{(A)}$,
one has
\begin{equation}
\label{eq:lower-bound-prop-corners}
|\xi-\zeta|\ge c_P A- 2.
\end{equation}
Moreover, for every $k\in\mathbb Z^2$,
\begin{equation}
\label{eq:multiplicity-prop-corners}
\#\Bigl\{
(\xi,\zeta)\in
\Gamma_{i,R}^{(A)}\times \Gamma_{j,R}^{(A)}
:
\zeta-\xi=k
\Bigr\}
\lesssim_P 1.
\end{equation}
\end{lemma}

\begin{proof}
Since $\Gamma_{j,R}^{(A)}\subset \widetilde\Gamma_{j,R}^{(A)}$, it suffices to prove the estimates for the $\widetilde\Gamma_{j,R}^{(A)}$s. Let $\xi\in \widetilde\Gamma_{i,R}^{(A)}$ and $\zeta\in \widetilde\Gamma_{j,R}^{(A)}.$ Define $x:=\pi_{i,R}\xi\in RS_i$ and $y:=\pi_{j,R}\zeta\in RS_j.$ Then $|\xi-x|\le 1$, $|\zeta-y|\le 1$, and $x,y$ lie at tangential distance at least $A$ from the endpoints of
their respective dilated sides.

We first prove \eqref{eq:lower-bound-prop-corners}. Since $S_i$ and $S_j$ are not parallel, there are two cases to consider.

\begin{itemize}
    \item We first consider the case where $S_i$ and $S_j$ are adjacent. Let $v$ be their common vertex. We have $a := |x-Rv|\ge A$ and $b:= |y-Rv|\ge A.$ Let $\theta_{ij} > 0$ denote the angle between the two side directions. Set
$$ \theta_P
:= \min_{\substack{S_i,S_j\ \mathrm{adjacent}}}
\theta_{ij}
>0.$$
By the law of cosines, 
\begin{align*}
|x-y|^2 &= a^2+b^2-2ab\cos\theta_{ij} = (a-b)^2 + 2ab(1-\cos\theta_{ij})\\ &\ge 2A^2(1- \cos\theta_{ij}) = 4A^2\sin^2\left(\frac{\theta_{ij}}{2}\right).    
\end{align*}
Thus, in this case, $$|x-y|\ge c_{\mathrm{adj}}\, A,$$ where $c_{\mathrm{adj}} :=2\sin\left(\frac{\theta_P}{2}\right)>0$.

\item If $S_i$ and $S_j$ are not adjacent, then $\operatorname{dist}(S_i,S_j) =:d_{ij}>0.$ Thus,  $\operatorname{dist}(RS_i,RS_j)=R d_{ij} > 0$. Assume $\Gamma_{i,R}^{(A)} \ne \varnothing$ and $\Gamma_{j,R}^{(A)} \ne \varnothing$. Then, $R|S_i|\ge 2A$ and $R|S_j|\ge 2A.$
Hence, $R\ge \frac{2A}{\max_m |S_m|}$, and 
$$|x-y| \ge Rd_{ij} \ge \frac{2A\, d_{ij}}{\max_m |S_m|}.$$
Therefore, in this case, $$|x-y|\ge c_{\mathrm{sep}}A,$$
where
$$ c_{\mathrm{sep}}
:= \min_{\substack{i,j\\ S_i,S_j\\\text{non-adjacent, non-parallel}}}
\frac{2d_{ij}}{\max_m |S_m|} > 0.$$ 
\end{itemize}

Now set $c_P:=\min\{c_{\mathrm{adj}},c_{\mathrm{sep}}\}>0.$  Then in both cases, $|x-y|\ge c_PA.$ By the triangle inequality,
$$|\xi-\zeta| \ge |x-y|-|\xi-x|-|\zeta-y|
\ge c_P A-2.$$

It remains to prove \eqref{eq:multiplicity-prop-corners}. Fix $k\in\mathbb Z^2$ and suppose $\zeta-\xi=k$. Then, $$\xi\in \Gamma_{i,R}^{(A)}\cap \bigl(\Gamma_{j,R}^{(A)}-k\bigr).$$ The set $\Gamma_{i,R}^{(A)}$ is contained in a strip of width 2 around the line $R\ell_i$, and $\Gamma_{j,R}^{(A)}-k$ is contained
in a strip of width 2 around the line $R\ell_j-k$. Since $S_i$ and $S_j$ are not parallel, the lines $\ell_i$ and $\ell_j$
meet at an angle bounded below by a positive constant depending only on $P$. Therefore the intersection of these two bounded-width strips has diameter bounded by a constant depending only on $P$. Thus, 
$$\#\left\{ \xi\in \mathbb Z^2: \xi\in \Gamma_{i,R}^{(A)}
\cap \bigl(\Gamma_{j,R}^{(A)}-k\bigr) \right\} \lesssim_P 1.$$ Equivalently,
$$ \#\left\{ (\xi,\zeta)\in \Gamma_{i,R}^{(A)}\times \Gamma_{j,R}^{(A)}
: \zeta-\xi=k \right\} \lesssim_P 1.$$
\end{proof}

Next, we discuss an almost-orthogonality statement.

\begin{lemma}
\label{orthogonality-lemma-1}
Let $\gamma\in L^\infty(\mathbb T^2)$. Let $\mathcal C_1,\ldots,\mathcal C_M$ be the equivalence classes of sides of $P$ under the relation $S_i\sim S_j$ if and only if $S_i$ is parallel to $S_j$. For each equivalence class $\mathcal C_\alpha$, define
$$\Gamma_{\alpha,R}^{(A)} := \bigcup_{S_j\in\mathcal C_\alpha}\Gamma_{j,R}^{(A)}.$$
Then there exists a function
$$
\ep(A)\to 0
\qquad\text{as }A\to\infty
$$
such that the following estimates hold uniformly in $R\ge 1$.
\begin{enumerate}
    \item Let $S_i$ and $S_j$ be two
    non-parallel sides of $P$. If 
    $\operatorname{supp}\widehat u\subset \Gamma_{i,R}^{(A)}$ and $\operatorname{supp}\widehat v\subset \Gamma_{j,R}^{(A)},$
    then
\begin{equation}
\label{eq: orthogonality}
\left| \int_{\mathbb T^2}\gamma(x)u(x)\overline{v(x)}\,dx
\right| \lesssim_P \ep(A)\, 
\|u\|_{L^2(\mathbb T^2)}\, \|v\|_{L^2(\mathbb T^2)}.
\end{equation}

    \item If $\operatorname{supp}\widehat u\subset \Gamma_{\alpha,R}^{(A)}$ and $\operatorname{supp}\widehat v\subset \Gamma_{\beta,R}^{(A)}\,$ for some $\alpha\ne\beta$, 
    then
\begin{equation}
\label{eq:class-orthogonality}
\left| \int_{\mathbb T^2}\gamma(x)u(x)\overline{v(x)}\,dx \right|
\lesssim_P \ep(A)\,\|u\|_2\, \|v\|_2.
\end{equation}
\end{enumerate}

\end{lemma}

\begin{proof} We prove (\ref{eq: orthogonality}). Set $E=\Gamma_{i,R}^{(A)}$ and $F=\Gamma_{j,R}^{(A)}$. We have the identity
$$
\int_{\mathbb T^2}\gamma(x) u(x)\overline{v(x)}\,dx
= \sum_{\xi\in E}\sum_{\zeta\in F}
\widehat u(\xi)\overline{\widehat v(\zeta)}
\,\widehat\gamma(\zeta-\xi).$$
Hence, by the Cauchy-Schwarz inequality,
$$\left|\int_{\mathbb T^2}\gamma(x) u(x)\overline{v(x)}\,dx\right|
\le \left(\sum_{\xi\in E,\zeta\in F}
|\widehat\gamma(\zeta-\xi)|^2
\right)^{1/2} \|u\|_2\|v\|_2.$$
By Lemma \ref{geometric-lemma-1}, every difference
$k=\zeta-\xi$ satisfies $|k|\ge c_P A-2$, and each such \(k\) occurs with multiplicity $\lesssim_P 1$. Therefore
$$ \sum_{\xi\in E,\zeta\in F} |\widehat\gamma(\zeta-\xi)|^2
\lesssim_P \sum_{|k|\ge c_P A-2} |\widehat\gamma(k)|^2.$$
Since $\gamma\in L^\infty(\mathbb T^2)\subset L^2(\mathbb T^2)$, we have $\widehat\gamma\in \ell^2(\mathbb Z^2)$. Thus the right-hand side tends to zero as \(A\to\infty\), independent of $R$. Defining $$ \ep(A) := \left(\sum_{|k|\ge c_P A-2} |\widehat\gamma(k)|^2 \right)^{1/2},$$
we obtain the desired estimate.

Note that every equivalence class $\mathcal{C}_\alpha$ in a convex polygon satisfies $\#\mathcal{C}_\alpha \le 2$. The second estimate (\ref{eq:class-orthogonality}) follows from (\ref{eq: orthogonality}) by decomposing $u$ and $v$ into their side pieces inside the two parallel classes. Since the polygon has finitely many sides, summing the pairwise estimates and applying Cauchy-Schwarz in the finite sums gives \eqref{eq:class-orthogonality}, with the same $\ep(A)$ up to a constant depending only on $P$.
\end{proof}

Next, we turn to the proof of Proposition \ref{polygon-without-corners}.

\begin{proof}[Proof of Proposition \ref{polygon-without-corners}]
    Let $\mathcal C_1,\ldots,\mathcal C_M$ be the equivalence classes of sides as in Lemma \ref{orthogonality-lemma-1}. For $A\ge 1$ and $R\ge 1$, define
    $$ \Gamma_{\alpha,R}^{(A)} := \bigcup_{S_j\in\mathcal C_\alpha}\Gamma_{j,R}^{(A)}.$$
    Then,
    $$ \Lambda_R^{(A)}(P) = \bigcup_{\alpha=1}^M \Gamma_{\alpha,R}^{(A)}.$$
    Let $g$ be a trigonometric polynomial satisfying $\operatorname{supp}\widehat g\subset \Lambda_R^{(A)}(P).$ Decompose $g$ as $g=\sum_{\alpha=1}^M g_\alpha$ where $\operatorname{supp}\widehat{g_\alpha} \subset \Gamma_{\alpha,R}^{(A)}.$ Since the frequency supports are disjoint,
\begin{equation}
\label{eq:orthogonality-equivalence-classes}
\|g\|_{L^2(\mathbb T^2)}^2
=
\sum_{\alpha=1}^M
\|g_\alpha\|_{L^2(\mathbb T^2)}^2.
\end{equation}
We first prove the observability estimate for each equivalence class $\mathcal{C}_\alpha$, where $\#\mathcal{C}_\alpha \le 2$ as noted earlier. Therefore, $\Gamma_{\alpha,R}^{(A)}$ is contained in the union of at most two strips of uniformly bounded width, both parallel to $\nu_\alpha^\perp$. By Theorem \ref{two-strip-theorem} applied in the direction $\nu_\alpha$, there exists a constant $C_\alpha<\infty$, independent of $R$ and $A$, such that
\begin{equation*}
    \|g_\alpha\|_{L^2(\mathbb T^2)}^2
    \le C_\alpha \int_{\mathbb T^2}|g_\alpha|^2\gamma.
\end{equation*}
Set $C_0:=\max_{1\le \alpha\le M} C_\alpha.$ Then 
\begin{equation}
    \label{eq:observability-per-class}
    \|g_\alpha\|_{L^2(\mathbb T^2)} \le C_0^{1/2} \left(\int_{\mathbb T^2}|g_\alpha|^2\gamma \right)^{1/2}
\end{equation}
for every $\alpha$. Now put $X_\alpha := \left(\int_{\mathbb T^2}|g_\alpha|^2\gamma \right)^{1/2}.$
Expanding the weighted norm of $g$ gives
$$ \int_{\mathbb T^2}|g|^2\gamma = \sum_{\alpha=1}^M X_\alpha^2 + 2\operatorname{Re} \sum_{1\le \alpha<\beta\le M} \int_{\mathbb T^2}\gamma g_\alpha\overline{g_\beta}.$$
For $\alpha\ne \beta$, Lemma \ref{orthogonality-lemma-1} gives
$$\left|\int_{\mathbb T^2}\gamma(x)\,  g_\alpha(x)\,\overline{g_\beta}(x)\, dx \right| \lesssim_P \ep(A)\, \|g_\alpha\|_2\, \|g_\beta\|_2.$$ 
Then \eqref{eq:observability-per-class} gives
$$\left|\int_{\mathbb T^2}\gamma(x)\,  g_\alpha(x)\,\overline{g_\beta}(x)\, dx \right| \lesssim_P \ep(A) \,C_0\, X_\alpha X_\beta.$$ 
Therefore,
$$ \int_{\mathbb T^2}|g|^2\gamma \ge \sum_{\alpha=1}^M X_\alpha^2
- 2\ep(A)\, C_0\, C_P \sum_{1\le \alpha<\beta\le M}X_\alpha X_\beta,$$
for some $C_P > 0$. Since 
$$ 2\sum_{1\le \alpha<\beta\le M}X_\alpha X_\beta
\le (M-1)\sum_{\alpha=1}^M X_\alpha^2,$$
we get
$$ \int_{\mathbb T^2}|g|^2\gamma \ge \left(1-\ep(A)\,C_0\, C_P\,(M-1)\right) \sum_{\alpha=1}^M X_\alpha^2.$$
Since $\ep(A)\to 0$ as $A\to\infty$, choose $A_0\ge 1$ large enough so that $\ep(A)\, C_0\, C_P\, (M-1)\le \frac12$ for all $A \ge A_0$. 
Then, $$\sum_{\alpha=1}^M X_\alpha^2 \le 2\int_{\mathbb T^2}|g|^2\gamma,$$
for every $A\ge A_0$. Combining this estimate with \eqref{eq:orthogonality-equivalence-classes} and \eqref{eq:observability-per-class}, we conclude that
$$ \|g\|_{L^2(\mathbb T^2)}^2 = \sum_{\alpha=1}^M
\|g_\alpha\|_{L^2(\mathbb T^2)}^2 \le C_0
\sum_{\alpha=1}^M \int_{\mathbb T^2}|g_\alpha|^2\gamma = 
C_0\sum_{\alpha=1}^M X_\alpha^2 \le 2C_0 \int_{\mathbb T^2}|g|^2\gamma.$$
Thus, Proposition \ref{polygon-without-corners} holds with $C_A(P,\gamma):=2C_0$ for every $A\ge A_0$. The constant is independent of $R$ as desired.
\end{proof}

\subsection{Uniqueness for Cone-Supported Fourier Series}
Finally, we discuss the proof of Proposition \ref{convex-cone}. The key ingredient is the following Hardy uniqueness result.

\begin{theorem}
\label{hardy-uniqueness-2}
    Let $F \in H^2(\D^2)$, and $F^* \in L^2(\T^2)$ denote its boundary function. If there exists a measurable set $E \subset \T^2$ with $|E| > 0$ such that $F^*$ vanishes on $E$, then $F = 0$ a.e. on $\D^2$. 
\end{theorem}

This follows from the well-known one-variable version through a Fubini-type argument.

\begin{theorem}[\cite{hoffman1962banach}]
\label{hardy-uniqueness-1}
    Let $h\in H^2(\D)$ and $h^* \in L^2(\T)$ denote its boundary function. If there exists a measurable set $E \subset \T$ with $|E| > 0$ such that $h^*(\xi) = 0$ for a.e. $\xi \in E$, then $h = 0$ a.e. on $\D$. 
\end{theorem}

\begin{proof}[Proof of Theorem \ref{hardy-uniqueness-2}]
Since $F$ is analytic, we can write
$$
F(z_1,z_2)
=
\sum_{m,n\ge0}c_{m,n}z_1^mz_2^n,
$$
and $F\in H^2(\D^2)$ means
$$
\sum_{m,n\ge0}|c_{m,n}|^2<\infty.
$$
Set $z_1=e^{2\pi i x_1}$ and $z_2=e^{2\pi i x_2}$. The boundary function is
$$
F^*(x_1,x_2)
=
\sum_{m,n\ge0}c_{m,n}e^{2\pi i(mx_1+nx_2)},
$$
with convergence in $L^2(\T^2)$. In particular, $F^*\in L^2(\T^2)$.

For each $n\ge0$, define
$$ \phi_n(x_1)
:= \sum_{m\ge0}c_{m,n}e^{2\pi i m x_1}. $$
Then $\phi_n\in H^2(\T)$ and by Plancherel,
$$\|\phi_n\|_{L^2(\T)}^2
= \sum_{m\ge0}|c_{m,n}|^2.$$
Therefore
$$\int_{\T}\sum_{n\ge0}|\phi_n(x_1)|^2\,dx_1
= \sum_{n\ge0}\|\phi_n\|_{L^2(\T)}^2
= \sum_{m,n\ge0}|c_{m,n}|^2 <\infty.
$$
It follows that for a.e. $x_1\in\T$,
$$
\sum_{n\ge0}|\phi_n(x_1)|^2<\infty.
$$
For every such $x_1$, the slice $x_2\mapsto F^*(x_1,x_2)$
belongs to $H^2(\T)$. Now assume that $F^*=0$ on a positive measure subset $E\subset\T^2$. For $x_1\in\T$, define
$E_{x_1} := \{x_2\in\T:(x_1,x_2)\in E\}.$ By Fubini's theorem,
$|E| = \int_{\T}|E_{x_1}|\,dx_1.$
Hence the set $E_1 := \{x_1\in\T:|E_{x_1}|>0\}$
has positive measure.

Fubini's theorem also implies that $F^*(x_1,x_2)=0$ for a.e. $x_2\in E_{x_1}$ and for a.e. $x_1\in E_1$. Let $E_1'\subset E_1$ be the set of those $x_1$ for which both of the following hold:
the slice $x_2\mapsto F^*(x_1,x_2)$ belongs to $H^2(\T)$, and
$F^*(x_1,x_2)=0$ for a.e. $x_2\in E_{x_1}$. By the preceding paragraphs, $|E_1'|>0$. For every $x_1\in E_1'$, the one-variable function $x_2\mapsto F^*(x_1,x_2)$ belongs to $H^2(\T)$ and vanishes on the positive-measure set $E_{x_1}$. By Theorem \ref{hardy-uniqueness-1}, it follows that
$F^*(x_1,x_2)=0$ for a.e. $x_2\in\T$. Thus $F^*(x_1,x_2)=0$ for a.e.
$(x_1,x_2)\in E_1'\times\T$. By the same coefficient argument with the two variables interchanged, for a.e. $x_2\in\T$, the slice $x_1\mapsto F^*(x_1,x_2)$ belongs to $H^2(\T)$.

Since $F^*=0$ for a.e. point of $E_1'\times\T$, Fubini's theorem gives that
$F^*(x_1,x_2)=0$ for a.e. $x_1\in E_1'$ and for a.e. $x_2\in\T$. For a.e. such $x_2$, the horizontal slice $x_1\mapsto F^*(x_1,x_2)$ belongs to $H^2(\T)$. Since $|E_1'|>0$, another application of Theorem \ref{hardy-uniqueness-1} gives $F^*(x_1,x_2)=0$ for a.e. $x_1\in\T$ and for a.e. $x_2\in\T$. Therefore $F^*=0$ a.e. on $\T^2$.

Finally, since
$$
F^*(x_1,x_2)
=
\sum_{m,n\ge0}c_{m,n}e^{2\pi i(mx_1+nx_2)}
$$
in $L^2(\T^2)$, all Fourier coefficients $c_{m,n}$ vanish. Hence, $F\equiv0$ in $H^2(\D^2)$.
\end{proof}

\begin{proof}[Proof of Proposition \ref{convex-cone}]

    It suffices to prove the result when the vertex of the cone is the origin. Indeed, suppose \(C=a+C_0\) where \(C_0\) is a closed convex cone with vertex at \(0\), and opening angle \(<\pi\). As $C_0$ is the intersection of two closed halfspaces, we can write $$
C_0=\{x\in\mathbb R^2:\ell_1(x)\ge 0,\ \ell_2(x)\ge 0\},$$
where $\ell_1, \ell_2$ are linearly independent linear functionals. Then,
$$
C=\{x\in\mathbb R^2:\ell_1(x)\ge \ell_1(a),\ \ell_2(x)\ge \ell_2(a)\}.$$

Choose \(\eta\in\mathbb Z^2\) such that $\ell_1(\eta)\le \ell_1(a)$ and $\ell_2(\eta)\le \ell_2(a).$ Then, for every \(\xi\in C\cap\mathbb Z^2\), we have $\ell_j(\xi-\eta)\ge 0$ for $j = 1,2$ and therefore \(\xi-\eta\in C_0\). Hence, $$(C\cap\mathbb Z^2)-\eta\subset C_0\cap\mathbb Z^2.$$ The reduction to the vertex-at-origin case from here is immediate.
Thus, we consider $f\in L^2(\T^2)$ with $\operatorname{supp} \widehat{f} \subset C \cap \Z^2$ such that $f = 0$ a.e. on a positive measure set $E\subset \T^2$. Since the opening angle of the cone $C$ is smaller than $\pi$, the dual cone $$C^* := \{x \in \R^2: x\cdot \xi \ge 0 \text{ for all } \xi \in C\},$$ has non-empty interior. Choose two linearly independent vectors $u,v \in \operatorname{int} C^*$. Without loss of generality, we may assume $u,v\in \Z^2$: this is possible because $\Q^2$ is dense in $\R^2$ and we can clear the common denominator to get integer vectors. Define the matrix $$A = \begin{pmatrix}
        u_1 & u_2\\ v_1 & v_2
    \end{pmatrix}.$$
    Then, for $\xi \in \Z^2$, we have $A\xi = \begin{pmatrix}
        u\cdot \xi \\ v\cdot \xi
    \end{pmatrix} \in \Z^2$. If there exists $\xi \in C \cap \Z^2$ such that $A\xi = (m,n)$, then set $c_{m,n} = \widehat{f}(\xi)$. If there is no such $\xi$, let $c_{m,n} = 0$. The $c_{m,n}$'s are well-defined since $A$ is injective. As $f\in L^2(\T^2)$, we have $\sum_{m,n\ge 0} |c_{m,n}|^2 = \sum_{\xi\in C\cap \Z^2} |\widehat{f}(\xi)|^2 < \infty$. This gives us a function
    $$F(z_1,z_2) := \sum_{m,n\ge 0} c_{m,n}\, z_1^m z_2^n = \sum_{\xi\in C \cap \Z^2} \widehat{f}(\xi)\, z_1^{u\cdot \xi} z_2^{v\cdot \xi} \in H^2(\D^2).$$
    Let $F^* \in L^2(\T^2)$ be the boundary function of $F$, i.e.,
    $$F^*(y_1,y_2) = \sum_{\xi\in C\cap \Z^2} \widehat{f}(\xi)\, e^{2\pi i ((u\cdot \xi)y_1 + (v\cdot \xi)y_2)} = \sum_{\xi\in C\cap \Z^2} \widehat{f}(\xi)\, e^{2\pi i \xi \cdot A^T y},$$
    for $0\le y_1,y_2 < 1$. Thus, $F^*(y) = f(A^T y)$. 

    Consider the map $\Phi: \T^2 \to \T^2$ given by $\Phi(y) = A^T y$. By the uniqueness of the Haar probability measure, we get $|\Phi^{-1}(E)| = |E| > 0$. Since $f = 0$ a.e. on $E$, we get $F^*(y) = f(\Phi(y)) = 0$ for a.e. $y\in \Phi^{-1}(E)$. The set $\Phi^{-1}(E)$ has positive measure, so Theorem \ref{hardy-uniqueness-2} gives $F = 0$ a.e. Thus, $c_{m,n} = 0$ for all $m,n \ge 0$. In particular, $\widehat{f}(\xi) = 0$ for every $\xi \in C \cap \Z^2$, and we conclude $f \equiv 0$.
\end{proof}

\appendix

\section{From Annular Observability to Exponential Stabilization}
\label{appendix:resolvent}

No claim of originality is made for the material in this section. AI was used to generate a rapid prototype, which was subsequently proofread and modified by
the authors.

We explain how the annular observability estimate of Theorem \ref{thm:annular-observability} implies exponential stabilization for the damped wave equation
$$\partial_t^2u-\Delta u+\gamma(x)\partial_tu=0,$$
on $\mathbb T^2$. The argument proceeds by deriving a uniform resolvent estimate and then applying the Gearhart--Pr\"uss theorem
\cite{gearhart78,pruss84}. Similar arguments appear in
\cite{green2019decay,green-jaye-mitkovski}. 

Define $$|D| := \frac{\sqrt{-\Delta}}{2\pi}.$$
Thus, $|D|e^{2\pi i k\cdot x} = |k|e^{2\pi i k\cdot x}$ and $-\Delta = 4\pi^2|D|^2.$

\subsection{A resolvent estimate}

We begin by converting the annular observability estimate into a
resolvent estimate for $|D|$.

\begin{proposition}
\label{prop:scalar-resolvent}
Let $\gamma\in H^{1/2}(\mathbb T^2)\cap L^\infty(\mathbb T^2)$ be non-negative and satisfy the GGCC \eqref{eq:GGCC-intro}. Then there exist constants $C>0$ and $\tau_0>0$ such that
$$ \|w\|_{L^2(\mathbb T^2)}^2 \le C\|(|D|-\tau)w\|_{L^2(\mathbb T^2)}^2 + C\int_{\mathbb T^2}\gamma(x)|w(x)|^2\,dx $$
for every $\tau\ge\tau_0$ and every $w\in H^1(\mathbb T^2)$.
\end{proposition}

\begin{proof}
Fix $\tau>0$. Let $\Pi_\tau$ denote the spectral projection of $|D|$ onto the interval $\left[\tau-\frac12,\tau+\frac12\right]$.
Thus,
$$ \operatorname{supp} \widehat{\Pi_\tau w} \subset \left\{ k\in\mathbb Z^2: \tau-\frac12 \le |k| \le \tau+\frac12 \right\}.$$
If $R_\tau:=\tau-\frac12$, then $\operatorname{supp}
\widehat{\Pi_\tau w} \subset A_{R_\tau,0}\cap\mathbb Z^2.$ Theorem \ref{thm:annular-observability}, with $\alpha=0$, gives
$$\|\Pi_\tau w\|_2^2
\le C_{\text{obs}} \int_{\mathbb T^2}
\gamma|\Pi_\tau w|^2,$$
where $C_{\text{obs}}$ is independent of $\tau$. Set
$$Q_\tau:=I-\Pi_\tau.$$
We have
$$\bigl||k|-\tau\bigr|
\ge \frac12,$$
on $\operatorname{supp} \widehat{Q_\tau w}$. Hence, by Plancherel's theorem,
$$
\|Q_\tau w\|_2^2
\le
4\|(|D|-\tau)Q_\tau w\|_2^2
\le
4\|(|D|-\tau)w\|_2^2.
$$
We now estimate the resonant part. Since $\Pi_\tau w=w-Q_\tau w$, we have
$$|\Pi_\tau w|^2 \le 2|w|^2+2|Q_\tau w|^2.$$
Therefore,
$$
\int_{\mathbb T^2}
\gamma|\Pi_\tau w|^2
\le
2\int_{\mathbb T^2}\gamma|w|^2
+
2\|\gamma\|_\infty\|Q_\tau w\|_2^2.
$$
Using the preceding estimate for $Q_\tau w$, we obtain
$$
\|\Pi_\tau w\|_2^2
\le 2C_{\text{obs}}\,\int_{\mathbb T^2}\gamma|w|^2
+ 8C_{\text{obs}}\,\|\gamma\|_\infty \|(|D|-\tau)w\|_2^2.
$$
Finally, $$ \|w\|_2^2 = \|\Pi_\tau w\|_2^2+\|Q_\tau w\|_2^2,$$
and hence
$$\|w\|_2^2 \le C\|(|D|-\tau)w\|_2^2 +
C\int_{\mathbb T^2}\gamma|w|^2,$$
for some $C > 0$.
\end{proof}

\subsection{The damped wave generator}

We next formulate the damped wave equation as a first-order equation. The energy is
$$ E_u(t) = \frac12 \left( \|\nabla u(t)\|_2^2 + \|\partial_tu(t)\|_2^2 \right).$$
The natural energy space is $\mathcal H := \dot H^1(\mathbb T^2)\times L^2(\mathbb T^2)$ where $\dot H^1(\mathbb T^2) := H^1(\mathbb T^2)/\mathbb C$. We say that $[u] = [\tilde u]$ for $u, \tilde u \in H^1(\T^2)$ if and only if $u - \tilde u \in \C$.  We equip $\mathcal H$ with the norm
$$
\|(u,v)\|_{\mathcal H}^2
:=
4\pi^2\||D|u\|_2^2+\|v\|_2^2.
$$
Here and below, we suppress the equivalence-class notation for the first
component. Since
$$
4\pi^2\||D|u\|_2^2
=
\|\nabla u\|_2^2,
$$
we have
$$
2E_u(t)
=
\|(u(t),\partial_tu(t))\|_{\mathcal H}^2.
$$
Define
$$
\mathcal A_\gamma
\begin{pmatrix}
u\\
v
\end{pmatrix}
=
\begin{pmatrix}
v\\
-4\pi^2|D|^2u-\gamma v
\end{pmatrix}
$$
with domain $D(\mathcal A_\gamma)
= H^2(\mathbb T^2)/\mathbb C \times
H^1(\mathbb T^2).$ Then, the damped wave equation is equivalent to
$$
\partial_t
\begin{pmatrix}
u\\
\partial_tu
\end{pmatrix}
=
\mathcal A_\gamma
\begin{pmatrix}
u\\
\partial_tu
\end{pmatrix}.
$$
We also introduce the free wave generator
$$
\mathcal A_0
\begin{pmatrix} u\\ v \end{pmatrix} = \begin{pmatrix} v\\ -4\pi^2|D|^2u
\end{pmatrix}. $$
The operator $\mathcal A_0$ is skew-adjoint on $\mathcal H$. Write $\mathcal A_\gamma
= \mathcal A_0+\mathcal K_\gamma$ where
$$
\mathcal K_\gamma
\begin{pmatrix}
u\\
v
\end{pmatrix}
=
\begin{pmatrix}
0\\
-\gamma v
\end{pmatrix}.
$$
Since $\gamma\in L^\infty(\mathbb T^2)$, the operator
$\mathcal K_\gamma$ is bounded on $\mathcal H$. Moreover, since
$\gamma\ge0$,
\begin{equation}
    \label{eq:dissipativity-generator}
    \operatorname{Re}
    \langle \mathcal K_\gamma U,U\rangle_{\mathcal H}
    =
    -\int_{\mathbb T^2}\gamma(x)|v(x)|^2\,dx
    \le0.
\end{equation}
Thus $\mathcal A_\gamma$ is maximal dissipative and hence, by the
Lumer--Phillips theorem, generates a contraction semigroup
$e^{t\mathcal A_\gamma}$ on $\mathcal H$. In particular,
$$
\|e^{t\mathcal A_\gamma}U_0\|_{\mathcal H}
\le
\|U_0\|_{\mathcal H},
$$
for $t\ge 0$.

\subsection{The high-frequency resolvent estimate}

The next step is to convert Proposition \ref{prop:scalar-resolvent} into
a resolvent estimate for the wave generator.

\begin{proposition}
\label{prop:high-frequency-generator}
There exist $C>0$ and $\tau_1>0$ such that
$$
\|U\|_{\mathcal H}
\le
C
\|(\mathcal A_\gamma-2\pi i\tau)U\|_{\mathcal H}
$$
for every $|\tau|\ge\tau_1$ and every
$U\in D(\mathcal A_\gamma)$.
\end{proposition}

\begin{proof}
Since $\gamma$ is real-valued, complex conjugation reduces the case
$\tau<0$ to the case $\tau>0$. We therefore suppose that
$\tau\ge \max\{\tau_0,1\}$.

Write
$$
G:=(\mathcal A_\gamma-2\pi i\tau)U
=\begin{pmatrix}f\\ g\end{pmatrix},
\qquad
U=\begin{pmatrix}u\\ v\end{pmatrix},
$$
and set
$$
r:=v+2\pi i|D|u,
\qquad
q:=v-2\pi i|D|u.
$$

A direct calculation gives
$$
2\pi(|D|-\tau)r
=
2\pi|D|f-i(g+\gamma v)
$$
and
$$
2\pi(|D|+\tau)q
=
2\pi|D|f+i(g+\gamma v).
$$

Set
$$
d:=\int_{\mathbb T^2}\gamma|v|^2.
$$
Since $|D|+\tau\ge\tau$ and $\tau\ge1$, the second identity gives
$$
\|q\|_2^2
\lesssim
\|G\|_{\mathcal H}^2+\|\gamma v\|_2^2
\lesssim
\|G\|_{\mathcal H}^2+d.
$$

Here we used
$$
\|\gamma v\|_2^2
\le \|\gamma\|_\infty d.
$$

Applying Proposition \ref{prop:scalar-resolvent} to $r$ and using the
first identity, we obtain
$$
\|r\|_2^2
\lesssim
\|G\|_{\mathcal H}^2+d
+\int_{\mathbb T^2}\gamma|r|^2.
$$

Since $r=2v-q$,
$$
\int_{\mathbb T^2}\gamma|r|^2
\lesssim
d+\|q\|_2^2
\lesssim
\|G\|_{\mathcal H}^2+d.
$$
Consequently,
$$
\|r\|_2^2+\|q\|_2^2
\lesssim
\|G\|_{\mathcal H}^2+d.
$$

On the other hand,
$$
\|r\|_2^2+\|q\|_2^2
=
2\|v\|_2^2+8\pi^2\||D|u\|_2^2
=
2\|U\|_{\mathcal H}^2.
$$
Thus
$$
\|U\|_{\mathcal H}^2
\lesssim
\|G\|_{\mathcal H}^2+d.
$$

Finally,
$$
d
=
-\operatorname{Re}\langle G,U\rangle_{\mathcal H}
\le
\|G\|_{\mathcal H}\|U\|_{\mathcal H}.
$$
Young's inequality and absorption now give
$$
\|U\|_{\mathcal H}
\lesssim
\|G\|_{\mathcal H}
=
\|(\mathcal A_\gamma-2\pi i\tau)U\|_{\mathcal H}.
$$
\end{proof}

\subsection{Bounded frequencies}

Proposition \ref{prop:high-frequency-generator} controls the resolvent for
large $|\tau|$. We next show that there are no obstructions on the
remaining bounded portion of the imaginary axis.

\begin{proposition}
\label{prop:imaginary-axis-resolvent}
Under the assumptions above, $2\pi i\tau\in\rho(\mathcal A_\gamma)$ for every $\tau\in\mathbb R$.
\end{proposition}

\begin{proof}
We first show that $\mathcal A_\gamma-2\pi i\tau$
is injective for every $\tau\in\mathbb R$. Suppose
$$
(\mathcal A_\gamma-2\pi i\tau)U=0,
$$
where $U=
\begin{pmatrix}
u\\ v \end{pmatrix}.$ Taking the real part of the $\mathcal H$ inner product with $U$ and
using \eqref{eq:dissipativity-generator}, we obtain
$$ \int_{\mathbb T^2}\gamma|v|^2=0.$$ 
Since $\gamma\ge0$, this implies $\gamma v=0$ almost everywhere. We first consider $\tau=0$. The first component of
$\mathcal A_\gamma U=0$ gives
$[v]=0$ in $\dot H^1(\mathbb T^2)$, so $v$ is constant. Since $\gamma$ satisfies GGCC, 
$$
\int_{\mathbb T^2}\gamma>0.
$$
Then, $\int_{\mathbb T^2}\gamma|v|^2=0$ forces $v=0$. The second component then gives $|D|^2u=0,$ so $u$ is constant. Hence $u=0$ in
$\dot H^1(\mathbb T^2)$, and therefore $U=0$ in $\mathcal H$.

Now suppose $\tau\ne0$. The first component gives
$ [v]-2\pi i\tau[u]=0$ in $\dot H^1(\mathbb T^2)$. Thus
$v-2\pi i\tau u$ is constant. By changing the representative of the equivalence class $[u]$ by an additive constant, we may therefore assume that $v=2\pi i\tau u.$ Since $\gamma v=0$ and $\tau\ne0$, it follows that $\gamma u=0$ almost everywhere.

The second component of $(\mathcal A_\gamma-2\pi i\tau)U=0$
is $$
-4\pi^2|D|^2u-\gamma v-2\pi i\tau v=0.
$$
Using
$v=2\pi i\tau u
$ and $\gamma u=0,
$ we obtain
$$
-4\pi^2|D|^2u
+
4\pi^2\tau^2u
=
0.
$$
Hence
$$
|D|^2u=\tau^2u.
$$
In terms of Fourier coefficients,
$$
(|k|^2-\tau^2)\widehat u(k)=0
$$
for every $k\in\mathbb Z^2$. Thus
$$
\operatorname{supp}\widehat u
\subset
\{k\in\mathbb Z^2:|k|=|\tau|\}.
$$
This is a finite set, so $u$ is a trigonometric polynomial. On the other hand, $\gamma u=0$ implies that $u=0$ almost everywhere on the positive-measure set
$\{\gamma>0\}.$ A nonzero trigonometric polynomial on $\mathbb T^2$ cannot vanish on a
set of positive measure. Therefore
$u\equiv0.$ It follows that $v=0$, and hence $U=0$. Thus,
$\mathcal A_\gamma-2\pi i\tau$ is injective for every $\tau\in\mathbb R$.

It remains to deduce surjectivity. The compact embedding
$D(\mathcal A_\gamma)\hookrightarrow\mathcal H$ shows that $\mathcal A_\gamma$ has compact resolvent. Consequently, by the
Fredholm alternative, every spectral point of $\mathcal A_\gamma$ is an eigenvalue. Since we have proved that
$\mathcal A_\gamma-2\pi i\tau$ is injective, it follows that
$2\pi i\tau\in\rho(\mathcal A_\gamma) $ for every $\tau\in\mathbb R$.
\end{proof}

We can now combine the high-frequency and bounded-frequency arguments.

\begin{proposition}
\label{prop:uniform-generator-resolvent}
Under the assumptions above,
$$
\sup_{\tau\in\mathbb R}
\left\|
(\mathcal A_\gamma-2\pi i\tau)^{-1}
\right\|_{\mathcal H\to\mathcal H}
<\infty.
$$
\end{proposition}

\begin{proof}
By Proposition \ref{prop:high-frequency-generator}, the resolvent is uniformly bounded for all sufficiently large $|\tau|$. By Proposition \ref{prop:imaginary-axis-resolvent},
$2\pi i\tau\in\rho(\mathcal A_\gamma)$ for every $\tau\in\mathbb R$. The resolvent map $z\mapsto(\mathcal A_\gamma-z)^{-1}$ is locally analytic, and in particular norm-continuous, on the resolvent set. Hence $\tau\mapsto (\mathcal A_\gamma-2\pi i\tau)^{-1}$ is norm-continuous on every bounded interval of $\mathbb R$. Its norm is therefore bounded on every compact interval. Combining the bounded-frequency estimate with the high-frequency estimate proves the claim.
\end{proof}

\subsection{Exponential stabilization}

We can now prove the stabilization statement from the Introduction.

\begin{corollary}[Exponential stabilization]
\label{cor:exponential-stabilization-appendix}
Let
$\gamma\in H^{1/2}(\mathbb T^2)\cap L^\infty(\mathbb T^2)$
be non-negative and satisfy the GGCC. Then there exist constants $C,c>0$ such that every solution of
$$
\partial_t^2u-\Delta u+\gamma(x)\partial_tu=0
$$
satisfies
$$
E_u(t)
\le
Ce^{-ct}E_u(0),
$$
for $t\ge0.$
\end{corollary}

\begin{proof}
As discussed above, $\mathcal A_\gamma$ generates a contraction semigroup
on the Hilbert space
$$
\mathcal H
=
\dot H^1(\mathbb T^2)\times L^2(\mathbb T^2).
$$
By Proposition \ref{prop:uniform-generator-resolvent}, we have
$2\pi i\mathbb R
\subset \rho(\mathcal A_\gamma)$
and $\sup_{\tau\in\mathbb R}
\left\| (\mathcal A_\gamma-2\pi i\tau)^{-1}
\right\|_{\mathcal H\to\mathcal H}
<\infty.$
Since $\{2\pi i\tau:\tau\in\mathbb R\}
= i\mathbb R,$ this is precisely a uniform resolvent bound on the imaginary axis. The Gearhart--Pr\"uss theorem \cite{gearhart78,pruss84} therefore implies
that there exist constants $M,\omega>0$ such that
$$
\|e^{t\mathcal A_\gamma}\|_{\mathcal H\to\mathcal H}
\le
Me^{-\omega t},
$$
for $t\ge 0$. Consequently,
$$
\|(u(t),\partial_tu(t))\|_{\mathcal H}
\le
Me^{-\omega t}
\|(u(0),\partial_tu(0))\|_{\mathcal H}.
$$
Squaring and recalling that
$$2E_u(t) = \|(u(t),\partial_tu(t))\|_{\mathcal H}^2,$$
we obtain
$$ E_u(t) \le M^2e^{-2\omega t}E_u(0),$$
for $t\ge 0$.
\end{proof}

\section*{Acknowledgments}

Research supported in part by NSF grants DMS-2453251 and DMS-2049477.  This research was undertaken in part while B.J. was a Simons Fellow.

AI was used for proofreading.

\nocite{*}
\bibliographystyle{plain}\bibliography{pls_2}

\end{document}